\documentclass{amsart}

\usepackage{amsmath,latexsym, amsthm}
\usepackage{amssymb}

\usepackage{amsfonts}
\usepackage{multirow}
\usepackage{multicol}
\usepackage{slashbox}

\usepackage{float}
\newtheorem*{putinar}{Archimedean Property}{\bf}{\it}
\newcommand*{\bfrac}[2]{\genfrac{(}{)}{0pt}{}{#1}{#2}}
\newcommand{\dsum}{\displaystyle\sum}
\newcommand{\dmin}{\displaystyle\min}
\newcommand{\dmax}{\displaystyle\max}

\def\K{\mathbf{K}}

\def\R{\mathbb{R}}
\def\N{\mathbb{N}}
\def\Q1{\mathbf{Q1}}

\def\B{\mathcal{B}}

\def\mrf{\mathbf{MOMRF}}

\def\Mo{\mathrm{M}}
\def\L{\mathrm{L}}

\def\y{\mathbf{y}}
\def\ll{\ell}

\newcommand{\fin}{\par \hfill
 $\square$  \par \bigskip}

\usepackage{pdflscape}
\usepackage[english, activeacute]{babel}
\usepackage[latin1]{inputenc}

\usepackage{hyperref}

\usepackage[left=3cm, right=3cm, top=3cm, bottom=3cm]{geometry}

\def\sym{\mathcal{S}}
\def\R{\mathbb{R}}

\def\B{\mathcal{B}}

\def\mrf{\mathbf{MOMRF}}

\def\Mo{\mathrm{M}}
\def\L{\mathrm{L}}

\def\N{\mathbb{N}}
\def\C{\mathbb{C}}

\def\e{\mathrm{e}}
\usepackage{amsaddr}
\usepackage{ytableau}
\def\Mo{\mathrm{M}}

\newtheorem{theorem}{Theorem}
\newtheorem{ex}[theorem]{Example}

\newtheorem{cor}[theorem]{Corollary}

\newtheorem{prop}[theorem]{Proposition}
\newtheorem{remark}[theorem]{Remark}
\newtheorem{lemma}[theorem]{Lemma}
\def\Q{\mathbf{Q}}
\def\y{\mathbf{y}}
\def\ll{\ell}
\usepackage{multicol}

\usepackage{bigstrut}
\usepackage{pgfplots}

\title[Algorithms for MF OMLP]{Algorithms and dimensionality reductions for continuous multifacility ordered median location problems}
\author[V. Blanco]{V\'ictor Blanco}
\address{Dpt. of Quantitative  Methods for Economics \& Business, Universidad de Granada}
\author[J. Puerto \and S. ElHaj-BenAli]{Justo Puerto \and Safae El-Haj Ben-Ali}
\address{Dpt. of Statistics and OR, Universidad de Sevilla}

\email{vblanco@ugr.es, puerto@us.es, anasafae@gmail.com}

\date{}
\begin{document}

\maketitle

\begin{abstract}
In this paper we propose a general methodology for solving a broad class of continuous, multifacility location problems, in any dimension and with $\ell_\tau$-norms proposing two different methodologies: 1) by a new second order cone mixed integer programming formulation and 2)  by formulating a sequence of semidefinite programs that converges to the solution of the problem; each of these relaxed problems solvable with  SDP solvers in polynomial time. We apply dimensionality reductions of the problems by sparsity and symmetry  in order to be able to solve larger problems.
\end{abstract}

\section{Introduction}

Multifacility location problems are among the most interesting and difficult problems in Location Analysis. It is well-known that even in their discrete version  the $p$-median and $p$-center problems are already NP-hard (see Kariv and Hakimi \cite{Kariv79a}.) A lot of attention has been paid in the last decades to these classes of problems, namely location-allocation problems, since they are easy to describe and to understand and they still  capture the essence of difficult problems in combinatorial optimization. A comprehensive overview over existing models and their applications is given in \cite{drezner02} and the references therein.
\medskip

On the other hand, also in the last two decades locators have devoted  much effort to solve continuous location problems that fall within the general class of global optimization, i.e.  convexity properties are lost.  Given a set of demand points (existing facilities) the goal is to locate several facilities to provide service to the existing ones  (demand points) minimizing some globalizing function of the travel distances. Assuming that each demand point will be served by its closest facility we are faced with another location-allocation problem but now the new facilities can be located anywhere in the framework space and therefore they are not confined to be in an ''a priori'' given set of locations. Obviously, these problems are much harder than the discrete ones and not much has been obtained regarding algorithms, and general convergence results, although some exceptions can be found in the literature \cite{BHMT00} and the references therein.
\medskip

Since the nineties a new family of objective functions has started to be considered in the area of Location Analysis: the ordered median problem \cite{NP05}. Ordered median problems represent as special cases nearly all classical objective functions in location theory, including the Median, CentDian, Center and $k$-Centra. Hence, handling the most important objective functions in location analysis is possible with one unique model and also new ones may be created by adapting adequately the parameters. More precisely, the $p$-facility ordered median problem  can be formulated as follows: A vector of weights $(\lambda_1,\ldots,\lambda_n)$ is given. The problem is to find  locations for the facilities that minimize the weighted sum of distances to the facilities where the distance to the closest point to its allocated facility is multiplied by the weight $\lambda_n$, the distance to the second closest, by $\lambda_{n-1}$, and so on. The distance to the farthest point is multiplied by $\lambda_1$.
As mentioned above, many location problems can be formulated as the ordered 1-median problem by selecting appropriate weights. For example, the vector for which all $\lambda_i= 1$ is the  $p$-median problem, the problem where $\lambda_1= 1$ and all others are equal to zero is the $p$-center problem, the problem where $\lambda_1=\ldots=\lambda_k=1$ and all others are equal to zero is the $p$-$k$-centrum. Minimizing the range of distances is achieved by $\lambda_1=1$, $\lambda_n= -1$ and all others are zero. Lots of results have been obtained for these problems in discrete settings, on networks and even in the continuous single facility case (see the book \cite{NP05}, \cite{BDNP06}, \cite{MSPV09} and the recent paper \cite{BPS12} ). However, very little is known in the continuous multifacility counterpart.
\medskip

In this paper, we address the multifacility continuous ordered median problem in finite dimension $d$ and for general $\ell_\tau$-norm for measuring the distances between points. We show how these problems can be cast within a  general family of polynomial optimization problems. Then, we show how these problems can be formulated as  second order cone mixed integer programs or, using tools borrowed from the Theory of Moments \cite{lasserrebook}, they can be solve (approximated up to any degree of accuracy) by a series of relaxed problems each one of them is a simple SDP that can be solved in polynomial time. We present preliminary computational results and show how the sizes and accuracy of the results can be improved by exploiting some specific characteristic of these models, namely sparsity in the representing variables and symmetry \cite{sparse,lasserre-sym,waki}.

\section{Preliminaries}
\label{sec:prelim}

In this section we recall the main definitions and results  on Semidefinite Programming and the Theory of Moments that will be useful for the development through this paper. We use standard notation in those fields (see e.g. \cite{lasserrebook,bookSDP}).

Semidefinite programming (SDP) is relatively a new subfield of convex optimization and probably one of the most exciting development in mathematical programming, it is a particular case of conic programming when one considers the convex cone of positive semidefinite matrices, whereas linear programming considers the positive orthant, a polyhedral convex cone. SDP theoretically includes a large number of convex programming such as convex quadratic programming (QP), or second-order cone programming (SOCP). After polynomial time interior point methods for linear optimization were extended to solve SDP problems it has seen a great growth during the 1990s. The handbook \cite{bookSDP} provides an excellent coverage of SDP as well as an extensive bibliography covering the literature up to year 2000.

Let $\mathbf{S}^n$ be the space of real $n\times n$ symmetric matrices. Whenever $P,\;Q\;\in \mathbf{S}^n$, the notation $P\succeq Q$ (resp. $P \succ Q$) stands for $P-Q$ positive semidefinite (resp. positive definite). Also, the notation $\langle P,Q\rangle$ stands for ${\rm trace}(PQ)$.

In canonical form, a primal semidefinite program reads:

\begin{equation*} \tag{SDP-P}
\left\{
\begin{array}{lcc}
\min &\langle C,X \rangle\\
s.t.  &\langle A_i,X \rangle=b_i,\; \text{ for }i=1,\ldots, m, \\
      &X\succeq 0.\\
\end{array}
\right.
\end{equation*}

where $C,\;A_i\in\mathbf{S}^n$ for $i=1, \ldots, m,$ and $b \in \R^m$. Its dual can be defined to be

\begin{equation*} \tag{SDP-D}
\left\{
\begin{array}{lcc}
\min & b^{t} y\\
s.t.  &\dsum_{i=1}^{m} y_{i}A_{i}-C\succeq0,\\
      &y\in\R^m.
\end{array}
\right.
\end{equation*}

SDP duality is not always strong because of the nonlinear positive semidefinite constraint. To avoid duality gaps, we can require the problem and its dual to satisfy some qualification constraint.  The purpose of a constraint qualification is to ensure the existence of Lagrange multipliers at optimality in nonlinear problems. These multipliers are an optimal solution for the dual problem, and thus the constraint qualification ensures that strong duality holds: it is possible to achieve primal and dual feasibility with no duality gap. One common choice of constraint qualification is Slater's constraint qualification \cite{BV04}. It is usually easy to verify that it holds for an SDP problem; indeed, it suffices to exhibit an  interior point for the nonlinear domain of the problem.

Next, we describe the basic elements of the Theory of Moments to be used to approximate hard global optimization problems \cite{lasserrebook}.
We denote by $\R[x]$ the ring of real polynomials in the variables $x=(x_1,\ldots,x_d)$,  for $d \in \N$ ($d \geq 1$), and by $\R[x]_r \subset \R[x]$ the space of polynomials of degree at most $r \in \N$ (here $\N$ denotes the set of non-negative integers). We also denote by $\B = \{x^\alpha: \alpha\in\N^d\}$ a canonical basis of monomials for $\R[x]$, where $x^\alpha = x_1^{\alpha_1} \cdots x_d^{\alpha_d}$, for any $\alpha \in \N^d$.  Note that $\mathcal{B}_r = \{x^\alpha \in \mathcal{B}: \dsum_{i=1}^d \alpha_i \leq r\}$ is a basis for $\R[x]_r$.

For any sequence indexed in the canonical monomial basis $\B$, $\mathbf{y}=(y_\alpha)_{\alpha \in \N^d}\subset\R$, let $\L_\mathbf{y}:\R[x]\rightarrow\R$ be the linear functional defined, for any $f=\dsum_{\alpha\in\N^d}f_\alpha\,x^\alpha \in \R[x]$, as $\L_\mathbf{y}(f) :=
\dsum_{\alpha\in\N^d}f_\alpha\,y_\alpha$.

The \textit{moment} matrix $\Mo_r(\mathbf{y})$ of order $r$ associated with $\mathbf{y}$, has its rows and columns indexed by  the elements in the basis $\mathcal{B}=\{x^\alpha: \alpha \in \N^d\}$ and for two elements in such a basis, $b_1=x^\alpha, b_2=x^\beta$, $\Mo_r(\mathbf{y})(b_1,b_2) = \Mo_r(\mathbf{y})(\alpha,\beta)\,:=\,\L_\mathbf{y}(x^{\alpha+\beta})\,=\,y_{\alpha+\beta}$, for $\vert\alpha\vert,\,\vert\beta\vert\,\leq r$ (here $|a|$ stands for the sum of the coordinates of $a \in \N^d$). Note that the moment matrix of order $r$ has dimension $\bfrac{d+r}{d}\times\bfrac{d+r}{d}$ and that there are $\bfrac{d+2r}{d}$ $\mathbf{y}_\alpha$ variables.

For $g\in \R[x] \,(=\dsum_{\gamma \in \N^d}  g_{\gamma} x^\gamma$), the \textit{localizing} matrix $\Mo_r(g \mathbf{y})$ of order $r$ associated with $\mathbf{y}$ and $g$, has its rows and columns indexed by the elements in $\mathcal{B}$ and for $b_1=x^\alpha$, $b_2=x^\beta$,
$\Mo_r(g\mathbf{y})(b_1, b_2) = \Mo_r(g\mathbf{y})(\alpha,\beta):=\L_\mathbf{y}(x^{\alpha+\beta}g(x))=\dsum_{\gamma}g_\gamma y_{\gamma+\alpha+\beta}$, for $\vert\alpha\vert,\vert\beta\vert\,\leq r$.

Observe that a different choice for the basis of $\R[x]$, instead of the standard monomial basis, would give different moment and localizing matrices, although the results would be also valid.

The main assumption to be imposed when one wants to assure convergence of some SDP relaxations for solving polynomial optimization problems (see for instance \cite{lasserre22,lasserrebook}) is a consequence  of Putinar's results \cite{putinar} and it is stated as follows.
\begin{putinar}
\label{defput}
Let $\{g_{1}, \ldots, g_m\} \subset \mathbb{R}[x]$ and $\mathbf{K}:=\{x\in \mathbb{R}^{d}: g_{j}(x)\geq
0,\: j=1,\ldots ,m \}$ a basic
closed semi-algebraic set. Then, $\mathbf{K}$ satisfies Archimedean property if there exists $u\in
\mathbb{R}[x]$ such that:
\begin{enumerate}
\item $\{x: u(x)\geq 0\} \subset \R^d$ is compact, and
\item\label{putrep} $u\,=\,\sigma _{0}+\dsum_{j=1}^{m}\sigma _{j}\,g_{j}$, for some  $\sigma_1, \ldots, \sigma_m \in \Sigma \lbrack x]$. (This expression is usually called a \textbf{Putinar's representation} of $u$ over $\mathbf{K}$).
\end{enumerate}
Being $\Sigma[x]\subset\mathbb{R}[x]$  the subset of polynomials that are sums of squares.
\end{putinar}
Note that Archimedean property is equivalent to impose that the quadratic polynomial $u(x)=M-\dsum_{i=1}^d x_i^2$ has a Putinar's representation over $\mathbf{K}$  for some $M>0$.

We observe that Archimedean property implies compactness of $\mathbf{K}$. It is easy to see that Archimedean property holds if either $\{x:g_{j}(x)\geq 0\}$ is compact for some $j$, or all $g_{j}$ are affine and $\mathbf{K}$ is compact. Furthermore, Archimedean property is not restrictive at all, since any semi-algebraic  set $\mathbf{K}\subseteq \R^d$ for which is known that $\dsum_{i=1}^d x_i^2 \leq M$ holds for some $M>0$ and for all $x\in \mathbf{K}$, admits a new representation $\mathbf{K'} = \mathbf{K} \cup \{x\in \mathbb{R}^{d}: g_{m+1}(x):=M-\dsum_{i=1}^d x_i^2\geq 0\}$ that verifies  Archimedean property  (see Section 2 in \cite{lasserrebook}).

The importance of Archimedean property stems from  the link between such a condition with the semidefiniteness of the moment and localizing matrices  (see \cite{putinar} ). The use of this property for the particular problems that we deal with through this paper will be given in the next sections. A detailed presentation and an account of its implications can be found in \cite{lasserre1}.

\begin{theorem}[Putinar \protect\cite{putinar}]
\label{thput} Let $\{g_{1}, \ldots, g_m\} \subset \mathbb{R}[x]$ and $\mathbf{K}:=\{x\in \mathbb{R}^{d}: g_{j}(x)\geq
0,\: j=1,\ldots ,m \}$ satisfying Archimedean property. Then:
\begin{enumerate}
\item Any $f\in \mathbb{R}[x]$ which is strictly positive on $\mathbf{K}$ has a Putinar's representation over $\mathbf{K}$.
\item $\mathbf{y}=(y_{\alpha })$ has a representing measure on $\mathbf{K}$ if and only if $\Mo_{r}(\mathbf{y})\succeq 0$, and $\Mo_{r}(g_{j}\mathbf{y})\succeq 0$, for all $j=1,\ldots ,m$ and  $r\in \N$.
\end{enumerate}
\end{theorem}

\begin{prop}[Lasserre \cite{lasserre1}] \label{pr:lasserreconv}
Let $\mathbf{K}:=\{x\in \mathbb{R}^{d}: g_{j}(x)\geq
0,j=1,\ldots ,\ll \}\subset \mathbb{R}^{d}$ satisfy the Archimedean Property and let $%
p\in \R[X]$ be a polynomial.
Let $r \geq r_0:=\max  \{ \lceil \deg \frac{p}{2} \rceil, \lceil \deg \frac{g_1}{2}
\rceil, \ldots, \lceil \deg \frac{g_l}{2} \rceil\}$,
and consider the hierarchy of semidefinite relaxations
\begin{equation}
  Q_r:\quad\begin{array}{rcl}
   \multicolumn{3}{l}{\inf_{y} \L_\mathbf{y}(p)} \\
  \Mo_r(\mathbf{y}) & \succeq & 0 \, , \\
  \Mo_{r - \lceil \deg g_j / 2 \rceil}(g_j \, \mathbf{y}) & \succeq & 0 \, , \quad
1 \le j \le \ll \,\\
 \L_\mathbf{y}(y_0)=1
  \end{array}
\end{equation}
with optimal value denoted by $\inf Q_r$.

Then, the hierarchy of SDP-relaxations $\{(Q_r)_{r \ge r_0}\}$ is monotone non-decreasing and converges to $\rho^*:=\min_{x\in K} p(x)$.
\end{prop}

\section{The multiple allocation multifacility ordered median location problem  \label{s:mfac-conv}}

This section deals with  multifacility location models where more than one new
facility have to be located to improve the service for the demand
points.  Several results obtained in previous papers are
extended or reformulated with great generality giving a panorama
view of the geometric insights of  location theory.

In this section we start by considering some multifacility  ordered median problems already introduced in \cite{NP05} and \cite{RNPF00}.  We shall extend these models, originally considered only in dimension $2$ and with polyhedral norms to the more general case of dimension $d$ and any $\ell_{\tau}$-norm being $\tau \in \mathbb{Q}, \; \tau\ge1$ (here $\ell_\tau$ stands for the norm $\|x\|_{\tau} = \left(\sum_{i=1}^d |x_i|^\tau\right)^{\frac{1}{\tau}}$, for all $x\in \R^d$). Unlike the original approaches in \cite{NP05,RNPF00} where even for polyhedral norms there are proposed iterative algorithms for which polynomiality results can not be proven we shall follow on a different approach.
In this section, we provide efficient reformulations of these classes of multifacility continuous location models and apply tools borrowed from conic programming to prove that these problems can be polynomially solved in the above mentioned cases, namely in dimension $d$ and under polyhedral or $\ell_{\tau}$-norm to measure distances.

We are given a set of demand points $\{a_{1},  . . .  , a_{n}\}$ and three sets of scalars $\{\omega_{1},  . . .  , \omega_{n}\}$, $\omega_i\ge 0,\; \forall\; i\in \{1, \ldots, n\} $, $\{\lambda_{1} , . . . , \lambda_{n}\}$ where $\lambda_{1}\geq . . . \geq\lambda_{n}\ge 0$ and $\{\mu_{12}, \mu_{13}, \ldots, \mu_{p-1 p}\}$ with $\mu_{jj'}\ge 0$ for $j, j' \in \{1, \ldots, p\}$ and $j'>j$.

The elements $\omega_{i}$ are weights corresponding to the importance given to the existing facilities $a_{i}, i \in \{1,  . . .  ,n\}$ and depending
on the choice of the $\lambda$-weights we get different classes of problems.   The $\mu$-weights represent the penalty per distance unit given when locating two different facilities. We denote by $\mathcal{P}_n$ the set of permutations of the first $n$ natural numbers.

Natural extensions of the multifacility models considered in \cite{NP05,RNPF00} assume that one is looking for the location of $p$ new facilities rather than only one. In this formulation the new facilities are chosen to provide service to all the existing facilities minimizing an ordered objective
function. These ordered problems are of course harder to handle than the classical ones not considering ordered distances. To simplify the presentation we consider that the different demand points use the same norm to measure distances, although all our results extend further to the case of mixed norms.

Let us consider a set of demand points $\{ a_1, a_2,\ldots, a_n\}\subset \R^d $. We want to locate $p$ new facilities $X=\{x_1,x_2, \ldots,
x_p\}$ which minimize the following expression:
\begin{equation}
\label{FI} f_{\lambda}^{NI}(x_1,x_2, \ldots, x_p)= \dsum_{i=1}^n
\dsum_{j=1}^p \lambda_{ij} d_{(i)}(x_j) + \dsum_{j=1}^{p-1} \dsum_{j'=j+1}^p
\mu_{jj'} \|x_j-x_{j'} \|_{\tau},
\end{equation}
where for any $x\in \R^d$, $d_i(x)=\|a_i-x\|_{\tau}$ and $d_{(i)}(x)$ is the $i$-th element in the permutation of $(d_1(x),\ldots,d_n(x))$ such that $ d_{(1)}(x)\ge d_{(2)}(x)\ge \stackrel{\underbrace{i}}{\ldots} \ge {d_{(i)}(x)}\ge\ldots \ge  d_{(n)}(x)$. In this model, it is assumed that (see \cite{RNPF00})
\begin{equation}\label{lambdas}
\begin{array}{c}
 \lambda_{11} \ge \lambda_{21} \ge \ldots \ge \lambda_{n1}\ge 0 \\
 \lambda_{12} \ge \lambda_{22} \ge  \ldots \ge \lambda_{n2}\ge 0 \\
\ldots  \\
 \lambda_{1p} \ge \lambda_{2p} \ge \ldots \ge \lambda_{np}\ge 0;
\end{array}
\end{equation}
$\mu_{jj'}\ge 0$ for any $j,j'=1,\ldots,p$ and, as mention above,
$d_{(i)}(x_j)$ is the expression, which appears at the $i$-th
position in the ordered version of the list
\begin{equation}
\label{etim} L^{NI}_j := (w_1 \|x_j- a_1\|_{\tau}, \ldots, w_n
\|x_j - a_n\|_{\tau}) \quad \mbox{ for } j=1,2,\ldots,p.
\end{equation}
Note that in this formulation we assign the lambda parameters with
respect to each new facility, i.e., $x_j$ is considered to be
non-interchangeable with $x_i$ whenever  $i \neq j$. For this
reason we say that this model has non-interchangeable facilities.

The problem consists of:
\begin{equation*} \label{pro:Multi1}\tag{$\mathbf{LOCOMF-NI}$}
\rho_\lambda^{NI} \,:=\displaystyle\min_{x} \{f_{\lambda}^{NI}(x): x=(x_1,\ldots,x_p), \; x_j\in \R^d, \, \forall j=1,\dots,p\},
\end{equation*}

The reader should observe that  this is the extension of Problem  (3) in \cite[Section 4.1]{RNPF00}.
\begin{theorem}
The problem \ref{pro:Multi1} admits at least one solution.
\end{theorem}
\begin{proof}
We know that
$$
\dsum_{i=1}^n\dsum_{j=1}^p \lambda_{ij} d_{(i)}(x_j)=\dsum_{i=1}^n\max_{\sigma}\dsum_{j=1}^p \lambda_{ij} w_{\sigma(i)}\|x_j-a_{\sigma(i)}\|_{\tau},
$$
where $\sigma$ is a permutation of the set $\{1,2,...,p\}$. Therefore, the first part of the objective function is a sum of maxima of convex functions. Hence, it is a convex function. Thus, $f_{\lambda}^{NI}$ is a convex function as a sum of convex functions.

Next, suppose that we restrict to consider the problem where $x_j=x$ for all $j=1,...,p$. Assume that $x^*$ is a solution. Then, for any $x$ not optimal it must exist $i_x\in\{1,...,n\}$ such that
$$
\|x-a_{i_x}\|_\tau\geq\|x^*-a_{i_x}\|_\tau.
$$
Thus taking $x=0$, we get
\begin{eqnarray}
  \|x^*-a_{i_0}\|_\tau &\leq& \|a_{i_0}\|_\tau, \nonumber\\
  \Longrightarrow\qquad \|x^*\|_\tau-\dmax_{i=1,...,n}\|a_{i}\|_\tau &\leq& \dmax_{i=1,...,n}\|a_{i}\|_\tau, \nonumber\\
  \Longrightarrow\qquad\qquad\qquad\qquad\;\;\; \|x^*\|_\tau &\leq& 2\dmax_{i=1,...,n}\|a_{i}\|_\tau=M. \label{Mbound}
\end{eqnarray}
We denote by $\mathcal{X}$ the set $\{x\in\R^{p\times d}:\;\|x_j\|_\tau\leq M,\;\forall j=1,...,p\}$.
Our problem consists of minimizing the convex function $f_{\lambda}^{NI}$, thus it is continuous over $\mathcal{X}$ which is compact in $\R^{p\times d}$ and consequently by Weierstrass theorem, problem \ref{pro:Multi1} admits at least one solution.
\qed
\end{proof}

\begin{remark}
Unlike the single facility problem, a solution of the multifacility problem is not unique in general. The following example shows what may happen. Consider the two-facility problem with set of demand point $A=\{(0,0),(0,1),(1,1),(0,1)\}$. Then, any point $(x_1,x_2)\in\big([(0,0),(0,1)], [(1,1),(0,1)]\big)$ is an optimal solution.
\end{remark}

Next we prove that Problem \ref{pro:Multi1} can be equivalently written as the following problem what will allow us the development of an efficient algorithm based on the theory of semidefinite programming.

\begin{theorem} \label{th:repMF1}
Let $\tau=\frac{r}{s}$ be such that $r,s\in\N\setminus\{0\}$, $r\ge s$ and $\gcd(r,s)=1$. For any set of lambda weights satisfying $\lambda_{1j}\geq . . . \geq\lambda_{nj}\ge 0$ for all $j=1,\ldots,p$, Problem \eqref{pro:Multi1} is equivalent to

\begin{eqnarray}
\rho_\lambda^{NI} = \min& \displaystyle \sum_{i=1}^{n}\sum_{j=1}^p v_{ij}+\sum_{\ell=1}^{n}\sum_{j=1}^p w_{\ell j} +\sum_{j =1}^{p-1}\sum_{j'=j+1}^p  t_{j j'} & \label{genpb-n1} \\
s.t &  v_{ij}+w_{\ell j}\geq\lambda_{\ell j}u_{ij} , & \forall i,\ell=1,...,n,\; j=1,\ldots,p \label{eq:n1-1} \\
 & y_{ijk}-x_{jk}+a_{ik}\ge 0,& i=1,\ldots,n,\; j=1,...,p,\;k=1,\ldots,d \label{eq:n1-2}\\
 & y_{ijk}+x_{jk}-a_{ik}\ge 0,& i=1,\ldots,n,\; \; j=1,...,p,\;k=1,\ldots,d, \label{eq:n1-3}\\
& y_{ijk}^r\leq \varsigma_{ijk}^{s}u_{ij}^{r-s},& i=1,\ldots,n,\; \; j=1,...,p,\;k=1,\ldots,d, \label{eq:n1-4}\\
&\omega_{i}^{\frac{r}{s}}\dsum_{k=1}^d \varsigma_{ijk}\le u_{ij},& i=1,\ldots,n,\; j=1,\ldots,p \label{eq:n1-5}\\
 & z_{j j'k}-x_{jk}+x_{j' k}\ge 0,& j,j'=1,\ldots,p,\; k=1,...,d, \label{eq:n1-6}\\
 & z_{j j'k}+x_{jk}-x_{j' k}\ge 0,& j,j'=1,\ldots,p,\; k=1,...,d,\label{eq:n1-7}\\
& z_{j j'k}^r\leq \xi_{j j'k}^{s}t_{j j'}^{r-s},& j,j'=1,\ldots,p,\; k=1,\ldots,d, \label{eq:n1-8}\\
&\mu_{j j'}^{\frac{r}{s}}\dsum_{k=1}^d \xi_{j j'k}\le t_{j j'},& j ,j'=1,\ldots,p,\; \label{eq:n1-9}\\
& \varsigma_{ijk}\ge 0,&  i=1,\ldots,n,\; j=1,\ldots,p,\;  k=1,\ldots,d, \label{eq:n1-10}\\
& \xi_{j j' k}\ge 0 & j, j' =1,\ldots,p,\;  k=1,\ldots,d. \label{eq:n1-11}
\end{eqnarray}

Moreover, Problem (\ref{genpb-n1}) satisfies Slater condition and it can
be represented as a semidefinite program with $(np+p^2)(2d+1)+p^2$  linear inequalities   and  at most $4(p^2d+npd)\log r$ linear matrix inequalities.
\end{theorem}
\begin{proof}
Note that the condition $\lambda_{1j}\geq . . . \geq\lambda_{nj}$ for all $j=1,...,p$, allows us to write Problem (\ref{pro:Multi1})  as

\begin{equation}\label{pb:init}
        \displaystyle\min_{x\in\R^{dp}}\max_{\sigma\in \mathcal{P}_n}\quad\dsum_{i=1}^{n}\dsum_{j=1}^{p}\lambda_{ij}\omega_{\sigma(i)}\|x_j-a_{\sigma(i)}\|_{\tau}+\dsum_{j=1}^{p-1}\dsum_{j'=j+1}^{p}\mu_{jj'}\|x_j-x_{j'}\|_{\tau},
\end{equation}
Let us introduce the auxiliary variables $u_{ij}$ and $t_{jj'}$, $i=1,\ldots,n$ and $j,j'=1,\ldots,p$ to which we impose that $ u_{ij}\geq\omega_{i}\|x_j-a_{i}\|_{\tau}$ and $t_{jj'}\geq\mu_{jj'}\|x_j-x_{j'}\|_{\tau}$, to model the problem in a convenient form.

Now, for any permutation $\sigma \in \mathcal{P}_n$, let $u_{\sigma j}=(u_{\sigma(1)j},\ldots,u_{\sigma(n)j})$ for $j=1,...,p$. Moreover, let us denote by $(\cdot)$ the permutation that sorts any vector in nonincreasing sequence, i.e. $u_{(1)j}\ge u_{(2)j}\ge \ldots \ge u_{(n)j}$.
Using that $\lambda_{1j}\geq...\geq\lambda_{nj}$ and since  $u_{ij_{}}\ge 0$, for all $i=1,\ldots,n$ and $j=1,\ldots,p$ then
$$
\displaystyle \sum_{i=1}^{n}\sum_{j=1}^{p}\lambda_{ij}u_{(i)j}=\max_{\sigma\in \mathcal{P}_n}\displaystyle \sum_{i=1}^{n}\sum_{j=1}^{p}\lambda_{ij}u_{\sigma(i)j}.
$$

The permutations in $\mathcal{P}_n$ can be represented by the following binary variables
$$
p_{ijk}=\left\{
         \begin{array}{ll}
           1, & \mbox{ if } u_{ij} \mbox{ goes in position } k, \\
           0, & \mbox{ otherwise},
         \end{array}
       \right.
$$

imposing that they verify the following constraints:

\begin{equation}
\left\{
  \begin{array}{ll}
    \displaystyle \sum_{i=1}^{n}p_{ijk}=1, & \forall j=1,...,p,\;k=1,...,n, \\
    \displaystyle \sum_{k=1}^{n}p_{ijk}=1, & \forall i=1,...,n,\;j=1,...,p.
  \end{array}
\right.
\end{equation}

Next, combining the two sets of variables we obtain that the objective function of \eqref{pb:init} can be equivalently written as
\begin{equation}\label{sys0}
    \left\{
      \begin{array}{ll}
        \displaystyle \sum_{i=1}^{n}\sum_{j=1}^{p}\lambda_{ij}u_{(i)j} =&\max\displaystyle \sum_{i=1}^{n}\sum_{j=1}^{p}\sum_{k=1}^{n}\lambda_{ij}u_{ij}p_{ik} \\
        & s.t \quad \displaystyle \sum_{i=1}^{n}p_{ijk}=1, \; \forall j=1,...,p,\;k=1,...,n, \\
        & \qquad \displaystyle \sum_{k=1}^{n}p_{ijk}=1, \; \forall i=1,...,n,\;j=1,...,p,\\
        & \qquad p_{ijk}\in\{0,1\}.
      \end{array}
    \right.
\end{equation}
Now, we point out that for fixed $j$ i.e. $u_{1j},...,u_{nj}$, we have
\begin{equation}\label{sys1}
    \left\{
      \begin{array}{ll}
        \displaystyle \sum_{i=1}^{n}\lambda_{ij}u_{(i)j} =&\max\displaystyle \sum_{i=1}^{n}\sum_{k=1}^{n}\lambda_{ij}u_{ij}p_{ijk} \\
        & s.t \quad \displaystyle \sum_{i=1}^{n}p_{ijk}=1, \; \forall k=1,...,n, \\
        & \qquad \displaystyle \sum_{k=1}^{n}p_{ijk}=1, \; \forall i=1,...,n,\\
        & \qquad p_{ijk}\in\{0,1\}.
      \end{array}
    \right.
\end{equation}
The problem below is an assignment problem and its constraint matrix is totally unimodular, so that solving a continuous relaxation of the problem always yields an integral solution vector \cite{AMO1993}, and thus a valid permutation.
Moreover, the dual of the linear programming relaxation of (\ref{sys1}) is strong and also gives the value of the original binary formulation of (\ref{sys1}).
Hence, for fixed $j\in\{1,...,p\}$ and for any vector $u_{\cdot j}\in \R^{n}$, by using the dual of the assignment problem (\ref{sys1}) we obtain the following expression
\begin{equation}\label{sys2}
    \left\{
       \begin{array}{ll}
         \displaystyle \sum_{i=1}^{n}\lambda_{ij}u_{(i)j}=&  \min\displaystyle \sum_{i=1}^{n}v_{ij}+\sum_{l=1}^{n}w_{lj} \\
         &  s.t \quad v_{ij}+w_{lj}\geq\lambda_{lj}u_{ij},\; \forall i,l=1,...,n.
       \end{array}
     \right.
\end{equation}

Finally, we replace (\ref{sys2}) in (\ref{pb:init}) and we get

\begin{equation}\label{genpb1}
    \left\{
      \begin{array}{ll}
        \min\displaystyle \sum_{i=1}^{n}\sum_{j=1}^{p}v_{ij}+\sum_{l=1}^{n}\sum_{j=1}^{p}w_{lj}+\sum_{j=1}^{p-1}\sum_{j'=j+1}^{p}t_{jj'}\\
        s.t \quad v_{ij}+w_{lj}\geq\lambda_{lj}u_{ij} , & \forall i,l=1,...,n,\;j=1,...,p, \\
        \qquad u_{ij}\geq\omega_{i}\|x_j-a_{i}\|_{\tau},& i=1,...,n,\;j=1,...,p,\\
        \qquad t_{jj'}\geq\mu_{jj'}\|x_j-x_{j'}\|_{\tau},& j,j'=1,...,p.
      \end{array}
    \right.
\end{equation}
It remains to prove that each inequality $u_{ij}\geq\omega_{i}\|x_j-a_{i}\|_{\tau},\;i=1,...,n,\;j=1,...,p$ can be replaced by the system
\begin{eqnarray*}
 y_{ijk}-x_{jk}+a_{ik}\ge 0,&  k=1,...,d, \\
 y_{ijk}+x_{jk}-a_{ik}\ge 0,&  k=1,...,d,\\
y_{ijk}^r\leq \varsigma_{ijk}^{s}u_{ij}^{r-s},&  k=1,...,d, \\
\omega_{i}^{\frac{r}{s}}\dsum_{k=1}^d \varsigma_{ijk}\le u_{ij},& \\
 \varsigma_{ijk}\ge 0,& \forall\;  k=1,\ldots,d.
\end{eqnarray*}

Indeed, set $\rho=\frac{r}{r-s}$, then $\frac{1}{\rho}+\frac{s}{r}=1$. Let $(\bar x_{j},\bar u_{ij})$ fulfill the inequality  $u_{ij}\geq\omega_{i}\|x_j-a_{i}\|_{\tau}$. Then we have
\begin{eqnarray}
  \omega_{i}\|\bar x_j-a_{i}\|_{\tau}\leq \bar u_{ij} &\Longleftrightarrow& \omega_{i}\left(\sum_{k=1}^{d}|\bar x_{jk}-a_{ik}|^{\frac{r}{s}}\right)^{\frac{s}{r}}\leq \bar u_{ij}^{\frac{s}{r}}\bar u_{ij}^{\frac{1}{\rho}} \nonumber \\
   &\Longleftrightarrow& \omega_{i}\left(\sum_{k=1}^{d}|\bar x_{jk}-a_{ik}|^{\frac{r}{s}}\bar u_{ij}^{\frac{r}{s}(-\frac{r-s}{r})}\right)^{\frac{s}{r}}\leq \bar u_{ij}^{\frac{s}{r}}\nonumber \\
&\Longleftrightarrow& \omega_{i}^{\frac{r}{s}}\sum_{k=1}^{d}|\bar x_{jk}-a_{ik}|^{\frac{r}{s}}\bar u_{ij}^{-\frac{r-s}{s}}\leq \bar u_{ij} \label{eq1}
\end{eqnarray}

Then (\ref{eq1}) holds if and only if $\exists \varsigma_{ij}\in\R^d$, $\varsigma_{ijk}\ge 0,\; \forall k=1,...,d$ such that
$$
|\bar x_{jk}-a_{ik}|^{\frac{r}{s}}\bar u_{ij}^{-\frac{r-s}{s}}\leq \varsigma_{ijk},\quad\mbox{ satisfying }\quad \omega_{i}^{\frac{r}{s}}\sum_{k=1}^{d}\varsigma_{ijk}\leq \bar u_{ij}, $$
or equivalently,
\begin{equation} \label{eq:n1} |\bar x_{jk}-a_{ik}|^{r}\leq \varsigma_{ijk}^{s}\bar u_{ij}^{r-s},\quad \omega_{i}^{\frac{r}{s}}\sum_{k=1}^{d}\varsigma_{ijk}\leq \bar u_{ij}.
\end{equation}
Set $\bar y_{ijk}=|\bar x_{jk}-a_{ik}|$ and $\bar \varsigma_{ijk}=|\bar x_{jk}-a_{ik}|^{\tau} \bar u_{ij}^{-1/\rho}$. Then, clearly $(\bar x_j,\bar u_{ij},\bar y_{ij},\bar \varsigma_{ij})$ satisfies (\ref{eq:n1-2})-(\ref{eq:n1-5}) and (\ref{eq:n1-10}).

Conversely, let $(\bar x_j,\bar u_{ij},\bar y_{ij},\bar \varsigma_{ij})$ be a feasible solution of (\ref{eq:n1-2})-(\ref{eq:n1-5}) and (\ref{eq:n1-10}). Then, $\bar y_{ijk}\ge |\bar x_{jk}-a_{ik}|$ for all $i,j$ and by (\ref{eq:n1-4}) $\bar \varsigma_{ijk}\ge \bar y_{ijk}^{(\frac{r}{s})}u_{ij}^{-\frac{r-s}{s}}\ge |\bar x_{jk}-a_{jk}|^{\tau} \bar u_{ij}^{-\frac{r-s}{s}}$. Thus,
$$ \omega_i^{\frac{r}{s}}\sum_{k=1}^d |\bar x_{jk}-a_{jk}|^{\frac{r}{s}} \bar u_{ij}^{-\frac{r-s}{s}} \le \omega_i^{\frac{r}{s}} \sum_{k=1}^d \bar \varsigma_{ijk} \le \bar u_{ij},$$
which in turns implies that $\omega_i^{\frac{r}{s}}\dsum_{k=1}^d |\bar x_{jk}-a_{jk}|^{\frac{r}{s}} \le \bar u_{ij} \bar u_{ij}^{\frac{r-s}{s}}$ and hence, $\omega_i\|\bar x_j-a_i\|_\tau \le \bar u_{ij}$.
In the same way we prove that each inequality $t_{jj'}\geq\mu_{jj'}\|x_j-x_{j'}\|_{\tau},\;j,j'=1,...,p$ can be replaced by the system
\begin{eqnarray*}
 z_{j j'k}-x_{jk}+x_{j' k}\ge 0,& k=1,...,d,\\
 z_{j j'k}+x_{jk}-x_{j' k}\ge 0,& k=1,...,d,\\
 z_{j j'k}^r\leq \xi_{j j'k}^{s}t_{j j'}^{r-s},& k=1,\ldots,d, \\
 \mu_{j j'}^{\frac{r}{s}}\dsum_{k=1}^d \xi_{j j'k}\le t_{j j'},& \\
 \xi_{ijk}\ge 0,& \forall\;  k=1,\ldots,d.
\end{eqnarray*}
Next, we observe that each one of the inequalities $y_{ijk}^r\leq \varsigma_{ijk}^{s}u_{ij}^{r-s},\; \; k=1,\ldots,d$ (respectively $ z_{j j'k}^r\leq \xi_{j j'k}^{s}t_{j j'}^{r-s},\; k=1,\ldots,d$) can be transformed, according to \cite[Lemma 3]{BPS14}, into $4\log r$ linear matrix inequalities (respectively $4\log r$ linear matrix inequalities),  being then exactly representable as second order cone constraints or semidefinite constraints.

Finally, it is straightforward to check Slater condition, for instance,  for the system (\ref{genpb1}). Set  $v_{ij}=1$, $w_{lj}=\frac{3}{2}M\lambda_{lj}\dmax_i\omega_i$, $u_{ij}=\frac{3}{2}M\omega_i+2$ with $M>>0$ large enough and $t_{jj'}=2M\mu_{jj'}+1$ for $i,l=1,...,n$ and $j,j'=1,...,p$.
 \qed
\end{proof}

In the particular case where $\tau=1$ (namely $r=s=1$) or the considered norm is polyhedral, the above problem reduces to a standard linear problem and the number of variables and inequalities is reduced.

As a consequence of Theorem \ref{th:repMF1}, Problem (\ref{pro:Multi1})  can be solved in polynomial time for any dimension $d$, by solving its reformulation as the SDP problem (\ref{genpb-n1})-(\ref{eq:n1-9}). The reader may note that this is an important step forward with respect to the already stated complexity results (see e.g. \cite{RNPF00}). There, it is proven that these problems are polynomial in $\R^2$ and polyhedral norms. Here we extend this complexity result for any polyhedral or $\ell_{\tau}$-norm and in any finite dimension.

\begin{ex}
\label{ex1}
Consider the two-facility problem with set of four demand points
$$A=\{(9.46, 9.36), (8.93, 7.00), (2.20, 1.12), (1.33, 8.89)\}$$
(a subset of the 50-cities data set from \cite{eilon-watson}), and (randomly generated-) lambda weights:
\begin{align*}
\lambda_{11} = 147.31, & \lambda_{12}= 119.08\\
\lambda_{21} = 24.44, & \lambda_{22} =  0.56\\
\lambda_{31} = 24.16, & \lambda_{32} = 0.00\\
\lambda_{41} = 10.77, & \lambda_{42} =  0.00
\end{align*}
$\mu_{12} =0.56$ and norm $\ell_2$.

Therefore, the problem to be solved can be written as:
\begin{eqnarray*}
\min_{x_1,x_2 \in \R^2} &  147.31d_{(1)}(x_1)+ 24.44 d_{(2)}(x_1)+ 24.16d_{(3)}(x_1)+ 10.77d_{(4)}(x_1) + 119.08d_{(1)}(x_2)+ \\
                        & +0.56d_{(2)}(x_2)+ 0.00d_{(3)}(x_2)+0.00d_{(4)}(x_2)+0.56\|x_1-x_2\|_2
\end{eqnarray*}

Then we get as solution: $x_1^*=(5.24, 6.41)$ and $x_2^*=(5.61, 5.44)$, with objective value $f^* = 1704.55$. Figure \ref{fig1} shows the points and the solutions of the problem.

\tikzset{mark options={mark size=1.25pt}}

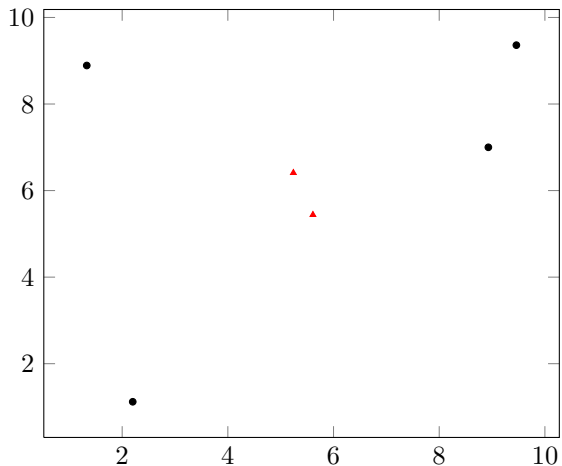
\begin{figure}
\label{fig1}
\begin{center}
\begin{tikzpicture}

\begin{axis}

\addplot[
scatter,
only marks,
point meta=explicit symbolic,
scatter/classes={
a={mark=*,black},%
w={mark=triangle*,red}
},
]
table[meta=label] {
x   y   label
9.46    9.36    a
8.93    7.00    a
2.20    1.12    a
1.33    8.89    a
5.24    6.41    w
5.61    5.44    w
};
\end{axis}
\end{tikzpicture}
\caption{Points in Example \ref{ex1} (filled circles) and solutions (triangles)}.
\end{center}
\end{figure}

\end{ex}

\section{Single Allocation Multifacility Location Problems with ordered median objective functions \label{s:locomf}}

The difference of the single allocation multifacility problems with those considered in the previous section rests on the fact that now each demand point shall be directed to a unique serving facility by means of a predetermined allocation rule (usually closest distance). This little difference makes the problem much more difficult since the convexity properties exhibited in the previous models are no longer valid and more sophisticated tools must be used to solve these problems.

In this framework, we are given a set $\{a_1,\ldots,a_n\}\subset \mathbb{R}^d$ endowed with a $\ell_{\tau}$-norm; and a feasible domain $\K=\{x \in \R^d: g_j(x)\geq 0, j=1, \ldots, m\}\subset \mathbb{R}^d$, closed and semi-algebraic. The goal is to find $p$ points $x_1, \ldots, x_p \in \mathbf{K}\subset \mathbb{R}^d$ minimizing some globalizing function of the shortest distances to  the set of demand points.

The main feature and what distinguishes multifacility location problems from other general purpose optimization problems, is that the dependence of the decision variables is given throughout the norms to the demand points, i.e. $\|x-a_i\|_{\tau}$.

For the ease of presentation we have restricted ourselves  to the particular case of pure location problem, namely $\tilde f_i(x):= \dmin_{j=1\ldots p} \|x_j-a_i\|_{\tau}$ which has attracted a lot of attention in the literature of location analysis. Needless to say that our methodology applies to more general forms of objective function, namely we could handle general rational functions of the distances as for instance in \cite{BPS12}.

We shall define the dependence of  the decision variables $x_1, \ldots, x_p \in \mathbb{R}^d$ via $t=(t_1, \ldots, t_n)$, where $t_{i}:\mathbb{R}^{pd}\mapsto \mathbb{R}$, $t_{i}(x_1, \ldots, x_p):= \dmin_j \|x_j-a_i\|_{\tau}$, $i=1,\ldots,n$. Therefore, the $i$-th component of the ordered  median objective function of our problems reads as
$$\begin{array}{llll} \tilde f_i(x):&\mathbb{R}^{pd} &\mapsto &\mathbb{R} \\
& x=(x_1,\ldots,x_p) &\mapsto &\displaystyle  t_i:=\min_{j=1..p} \{\|x_j-a_i\|_{\tau}\}.
\end{array} $$

Consider the following problem
\begin{equation*} \label{pro:LocPop}\tag{$\mathbf{LOCOMF}$}
\rho_\lambda \,:=\displaystyle\min_{x}\{\dsum_{i=1}^n \lambda_i \tilde f_{(i)}(x): x=(x_1,\ldots,x_p), \; x_j\in \mathbf{K}, \, \forall j=1,\dots,p\},
\end{equation*}%
where:

$\bullet$  $\K \subseteq \R^{d}$ satisfies the Archimedean property. Without loss of generality we shall assume that we know $M>0$ such that $\|x_j\|_2\le M$, for all $j=1,\ldots,p$.

$\bullet$ $\tau:=\frac{r}{s}\ge 1$, $r,s\in \mathbb{N}$ with $gcd(r,s)=1$.

$\bullet$  $\lambda_{\ell}\ge 0$ for all $\ell=1,\dots,n.$

First of all, we observe that problem \ref{pro:LocPop} is well defined and that it has optimal solution. Indeed, we are minimizing a continuous function over a compact set in $\R^{d}$. Thus, by Weierstrass theorem problem \ref{pro:LocPop} admits an optimal solution.

\subsection{A second order cone mixed integer programming approach to solve \ref{pro:LocPop}\label{ss:prolanda2}}

In this section, we present a tractable formulation of problem \ref{pro:LocPop} as a mixed integer nonlinear program with linear objective function.  For each $i \in \{1, \ldots, n\}$, we set $UB_i$ as a valid upper bound on the value of $\|\bar{x}_j-a_i\|_\tau$, $\bar{x}_j\in\K$.

We introduce the following auxiliary problem
\label{}
\begin{align}
 \quad \hat{\rho}_\lambda = \displaystyle \min & \sum_{\ell=1}^{n}\lambda_{\ell}\theta_{\ell}  \label{pro:locgen2}\tag{${\rm MFOMP}_\lambda$}\\
 \mbox{s.t.  } h_{il}^1&:=t_i\le \theta_{\ell} +UB_i(1-w_{i\ell}), & i=1,\ldots,n, \;  \ell=1,\ldots,n,\label{c:ti2} \\
 h_{l}^2&:=\theta_{\ell}\ge \theta_{\ell+1}, &  \ell=1,\ldots,n-1, \label{c:theta1} \\
 h_{ij}^3&:=u_{ij}\le t_i+UB_i(1-z_{ij}), &  \forall \; i=1,\ldots, n,\; j=1,\ldots,p,  \label{c:ti1} \\
 h_{ijk}^4&:=v_{ijk}-x_{jk}+a_{ik}\ge 0,& i=1,\ldots,n,\; j=1,\ldots,p,\;  k=1,...,d, \label{c:n1-1} \\
 h_{ijk}^5&:=v_{ijk}+x_{jk}-a_{ik}\ge 0,& i=1,\ldots,n,\; j=1,\ldots,p,\; k=1,..,d,  \label{c:n1-2}    \\
 h_{ijk}^6&:=v_{ijk}^r\leq \zeta_{ijk}^{s}u_{ij}^{r-s},& i=1,\ldots,n,\; j=1,\ldots,p,\; k=1,\ldots,d,  \label{c:n1-3}   \\
 h_{ij}^7&:=\sum_{k=1}^d \zeta_{ijk}\le u_{ij},& i=1,\ldots,n,\; j=1,\ldots,p, \label{c:n1-4}   \\
 h_{i}^8&:=\dsum_{j=1}^pz_{ij}=1,& \; i=1,\ldots,n,   \label{c:n1-5} \\
 h_{l}^9&:=\dsum_{i=1}^nw_{il}=1,& \; l=1,\ldots,n,  \label{c:n1-6}  \\
 h_{i}^{10}&:=\dsum_{l=1}^nw_{il}=1,& \; i=1,\ldots,n,\label{c:n1-7} \\
&w_{i\ell} \in \{0,1\}, & \; \forall\ i,\ell=1,\ldots,n,  \label{c:domain-w}\\
&z_{ij} \in \{0,1\}, &\; \forall\ i=1,\ldots,n, j=1, \ldots, p,  \label{c:domain-z}\\
& \theta_{\ell},\;t_i,\; v_{ijk},\;\zeta_{ijk},\;u_{ij} \in \R^+, & i,l=1, \ldots, n, j=1, \ldots, p, k=1,\ldots,d,   \label{c:domain-tvud} \\
&x_j\in \mathbf{K}, & j=1, \ldots, p.   \label{c:domain-x}
\end{align}

With constraints (\ref{c:ti2})-(\ref{c:theta1}) we enforce the variable $\theta_l$ to assume the value $t_i$ that is sorted in the $l$-th position of the vector $t$, while constraints (\ref{c:ti1})-(\ref{c:n1-4}) model the evaluation of $\|x_j-a_i\|_{\tau}$ for all $i$ and $j$.  Constraints (\ref{c:n1-6})-(\ref{c:domain-w}) model  permutations, and constraints (\ref{c:n1-5}) and (\ref{c:domain-z}) are introduced to model the allocation of element indexed by $i$ to a unique index $j$. Therefore, putting all the above ingredients together we get that in the optimum $t_i=\dmin_j\|x_j-a_i\|_{\tau}$.

Let us denote  by $\{h_1,\dots,h_{nc1}\}$ with $nc1:=3n+n^2+n-1+np(3d+2)=n^2+4n+np(3d+2)-1$ the constraints in the problem above, once excluded those defining $\mathbf{K}$. Let $\mathbf{\hat{K}}$ denote the feasible domain of Problem \ref{pro:locgen2}.

\begin{theorem} \label{t:equi-sl}
Let $x$ be a feasible solution of \ref{pro:LocPop} then there exists a solution $(x,z,u,v,\zeta,w, t,\theta)$ for \ref{pro:locgen2} such that their objective values are equal. Conversely, if $(x,z,u,v,\zeta,w, t,\theta)$ is a feasible solution for \ref{pro:locgen2} then  $x$ is a feasible solution for \ref{pro:LocPop}. Furthermore, if $K$ satisfies Slater condition then the feasible region of the continuous relaxation of \ref{pro:locgen2} also satisfies Slater condition and $\rho_{\lambda}=\hat{\rho}_{\lambda}$.
\end{theorem}
\begin{proof}
Let $\bar{x}=(\bar{x}_1,...,\bar{x}_p)$ be a feasible solution of \ref{pro:LocPop}. Then, it satisfies $\bar{x}_j\in\K$, for all $j=1,...,p$. Let $u_{ij}=\|\bar{x}_j-a_i\|_\tau$, based in (\ref{eq1}) and (\ref{eq:n1}),  $\|\bar{x}_j-a_i\|_\tau$ can be represented by
$$
\left\{
  \begin{array}{ll}
    v_{ijk}&=|\bar{x}_{jk}-a_{ik}|,\\
    v_{ijk}^r&=\zeta_{ijk}^su_{ij}^{r-s}, \\
    \dsum_{k=1}^d\zeta_{ijk}&=u_{ij}, \\
    \zeta&\geq0.
  \end{array}
\right.
$$
For $i=1,...,n,\;j=1,...,p$ and $k=1,...,d$, we denote by
$$\begin{array}{rl}
t_i=\dmin_{\ell}\|\bar{x}_{\ell}-a_i\|_\tau
&\quad\mbox{ and }\quad z_{ij}=\left\{\begin{array}{ll}
                1, & \mbox{if }\dmin_{l}\|\bar{x}_l-a_i\|_\tau=\|\bar{x}_j-a_i\|_\tau, \\
                0, & \mbox{otherwise.}
                \end{array},
         \right.
\end{array}$$

Observe that if it would exist $j'\in\{1,...,p\}$ such that $j'\neq j$ and $\dmin_{l}\|\bar{x}_l-a_i\|_\tau=\|\bar{x}_{j'}-a_i\|_\tau$ then we can choose arbitrarily  any of them, because a client can be assigned to only one facility.

These values clearly satisfy constraints (\ref{c:ti1})-(\ref{c:n1-4}), (\ref{c:domain-z}) and (\ref{c:domain-tvud}).

Besides, let $\sigma$ be the permutation of $(1,...,n)$ such that $t_{\sigma(1)}\geq...\geq t_{\sigma(n)}$. Take,
$$
w_{il}=\left\{
         \begin{array}{ll}
           1, & \mbox{if }i=\sigma(l), \\
           0, & \mbox{otherwize};
         \end{array}
       \right.
\mbox{ and } \theta_l=t_{\sigma(l)}.
$$

then the constraints (\ref{c:ti2})-(\ref{c:theta1}) are also satisfied. Clearly, $\sum_{\ell=1}^n \lambda_{\ell} \theta_{\ell}=\sum_{i=1}^n \lambda_i \tilde{f}_{(i)} (x)$.

Conversely, if $(\bar{x},\bar{z},\bar{u},\bar{v},\bar{\zeta},\bar{w},\bar{t},\bar{\theta})$ is a feasible solution of \ref{pro:locgen2} then, clearly $\bar{x}_j \in\K$ and $\bar{x}$ is a feasible point of \ref{pro:LocPop}.

Consider the continuous relaxation of problem (\ref{pro:locgen2}). Suppose that $\K$ satisfies Slater condition. Take $x_j$ for all $j=1,\ldots,p$, in the interior of $\K$. Set $\theta=4M+\frac{1}{\ell}$, $t_i=3M$ and $u_{ij}=2M$ for $\ell,i=1,...,n$ and $j=1,...,p$ with $M>>0$ large enough. Then, for any $z_{ij},w_{i\ell} \in [0,1]$ we get that the set of inequality constraints satisfies
$$
\left\{
  \begin{array}{ll}
    \theta_{\ell}-t_i +UB_i(1-w_{i\ell})>0, & i=1,\ldots,n, \;  \ell=1,\ldots,n,\\
    \theta_{\ell}> \theta_{\ell+1}, & \ell=1,\ldots,n, \\
    t_i-u_{ij}+UB_i(1-z_{ij})>0, &  \forall \; i=1,\ldots, n,\; j=1,\ldots,p,\\
    u_{ij}>\|x_j-a_i\|_\tau, &  \forall \; i=1,\ldots, n,\; j=1,\ldots,p.
  \end{array}
\right.
$$
This proves that the continuous relaxation of \ref{pro:locgen2} satisfies Slater condition.

Clearly, by the above arguments, optimal solutions and optimal values of both formulations coincide.
\qed
\end{proof}

\begin{ex}
\label{ex:ex2}
In the following example we have extracted the following $10$ points from the $50$-points data set in \cite{eilon-watson} to illustrate the applicability of the above formulation:\\
$(9.46, 9.36), (7.43, 1.61), (6.27, 3.66), (5.00, 9.00), (2.83, 9.88), (2.20, 1.12), (1.90, 8.35), (1.68, 6.45),$ \\
$ (1.24, 6.69), (0.75, 4.98).$

 For $p=3$, $\tau=\dfrac{7}{5}$ and $\lambda$-weights:
$$
2.25, 1.70, 1.14, 1.11, 1.06, 1.03, 1.01, 1.01, 1.00, 1.00,
$$
we get the solutions $x_1^*=(6.199838,1.580148)$, $x_2^*=(5.000041,9.360006)$, and $x_3^*=(1.440000,6.550015)$, with optimal objective value $f^*=30.1460$. Figure \ref{fig2} shows the demand points (filled dots), solutions (filled triangles) and the allocation of the demand points to the facilities (dashed lines).

\tikzset{mark options={mark size=1.25pt}}

\begin{figure}
\label{fig2}
\begin{center}
\begin{tikzpicture}

\begin{axis}

\addplot[
scatter,
only marks,
point meta=explicit symbolic,
scatter/classes={
a={mark=*,black},%
w={mark=triangle*,red},%
c={mark=o,blue}},
]
table[meta=label] {
x   y   label
9.460001	 9.359996 	 a
7.429997	 1.609995 	 a
6.269996	 3.659997 	 a
5.000001	 8.999998 	 a
2.829998	 9.880000 	 a
2.199999	 1.119997 	 a
1.900003	 8.350003 	 a
1.679995	 6.450000 	 a
1.239999	 6.690005 	 a
0.750001	 4.980004 	 a
6.199838	 1.580148 	 w
5.000041	 9.360006 	 w
1.440000	 6.550015 	 w
};
\end{axis}
\end{tikzpicture}
\begin{tikzpicture}

\begin{axis}

\addplot[
scatter,
only marks,
point meta=explicit symbolic,
scatter/classes={
a={mark=*,black},%
w={mark=triangle*,red},%
c={mark=o,blue}},
]
table[meta=label] {
x   y   label
9.460001	 9.359996 	 a
7.429997	 1.609995 	 a
6.269996	 3.659997 	 a
5.000001	 8.999998 	 a
2.829998	 9.880000 	 a
2.199999	 1.119997 	 a
1.900003	 8.350003 	 a
1.679995	 6.450000 	 a
1.239999	 6.690005 	 a
0.750001	 4.980004 	 a
6.199838	 1.580148 	 w
5.000041	 9.360006 	 w
1.440000	 6.550015 	 w
};
\addplot[mesh, dashed, black] coordinates {(9.460001,9.359996)  (5.000041,9.360006)};
\addplot[mesh, dashed, black] coordinates {(7.429997,1.609995)   (6.199838,1.580148)};
\addplot[mesh, dashed, black] coordinates {(6.269996,3.659997)   (6.199838,1.580148)};
\addplot[mesh, dashed, black] coordinates {(5.000001,8.999998)   (5.000041,9.360006)};
\addplot[mesh, dashed, black] coordinates {(2.829998,9.880000)   (5.000041,9.360006)};
\addplot[mesh, dashed, black] coordinates {(2.199999,1.119997)   (6.199838,1.580148)};
\addplot[mesh, dashed, black] coordinates {(1.900003,8.350003)   (1.440000,6.550015)};
\addplot[mesh, dashed, black] coordinates {(1.679995,6.450000)   (1.440000,6.550015)};
\addplot[mesh, dashed, black] coordinates {(1.239999,6.690005)   (1.440000,6.550015)};
\addplot[mesh, dashed, black] coordinates {(0.750001,4.980004)   (1.440000,6.550015)};

%
%
\end{axis}
\end{tikzpicture}

\caption{Points in Example \ref{ex:ex2} (filled circles), solutions (triangles) and allotation of demand points to facilities (dashed lines)}.
\end{center}
\end{figure}
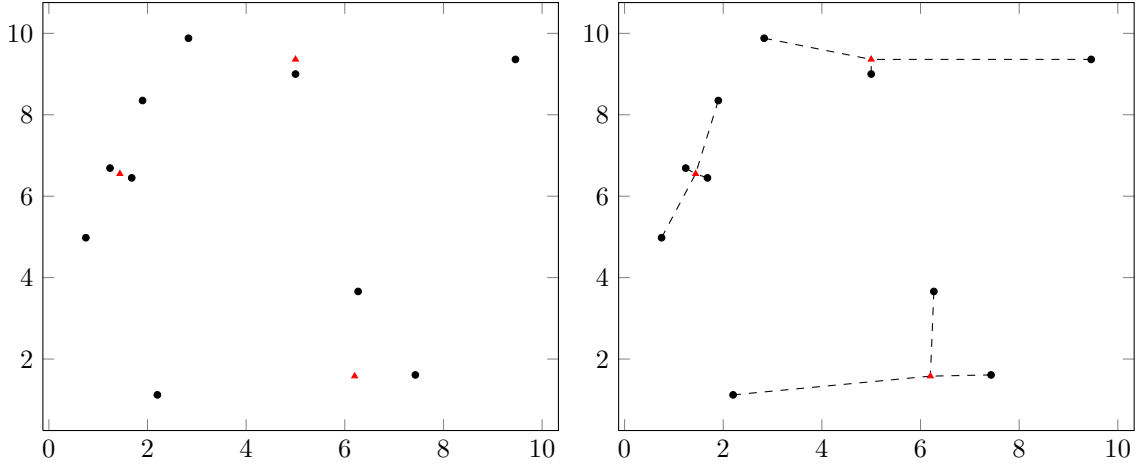

\end{ex} An interesting observation that follows from Problem (\ref{pro:locgen2}) is that the unconstrained version of the location problem  can be equivalently seen as a mixed integer second order cone program. Observe that the only nonlinear  constraints that appear are  (\ref{c:n1-3}). However, (\ref{c:n1-3}) can be written equivalently as a polynomial number of second order cone constraints, according to \cite{BPS14}.  This way, Problem (\ref{pro:locgen2}) becomes a  mixed integer nonlinear program with lineal objective function and only linear and second order cone constraints, although with two sets of binary variables, namely $w$ and $z$. Nevertheless, there are nowadays general purpose solvers, as Gurobi, Cplex or Xpress, that implements exact B\&B algorithms for this type of problems and that are rather efficient.

From the above observation   to solve Problem \ref{pro:locgen2} efficiently we have combined a branch-and-bound approach over a mixed integer nonlinear program. In our approach, we provide two types of lower bounds in the nodes of the branching tree: a continuous relaxation and a SPD relaxation based on a hierarchy of SDP ''\textit{a la Lasserre}''. Clearly, the first type of bounds are only possible if in each node of the tree the continuous relaxation of \ref{pro:locgen2} satisfies Slater condition. This is ensured by Theorem \ref{t:equi-sl}.

The consequence of the above transformation is that one can easily put this family of problems in commercial solvers and then get solutions without going to painful \textit{ad hoc} implementations that may be problem dependent. We illustrate this approach in our computational experiments in Section \ref{sec:tests}.

Finally, we conclude this section providing the second type of lower bounds, based on a SDP hierarchy,  that can be used to approximate to any degree of accuracy the solution of the problem as well as within the branch and bound framework at each node of the branching tree.

Let $\mathbf{y}=(y_\alpha)$ be a real sequence
indexed in the monomial basis $(x^\beta z^{\eta} u^\gamma v^\delta \zeta^\kappa w^\psi t^\tau \theta^\vartheta)$ of $\mathbb{R}[x,z,u,v,\zeta,w,t,\theta]$ (with $%
\alpha=(\beta,\eta,\gamma,\delta,\kappa,\tau,\psi,\vartheta)\in\mathbb{N}^{pd} \times \mathbb{N}^{np} \times \mathbb{N}^{np} \times \mathbb{N}^{npd} \times \mathbb{N}^{n^2}\times \mathbb{N}^{n}\times \mathbb{N}^{n} \times  \mathbb{N}^{npd}$). Denote by $nv_1=d+np(d+2)+n^2+2n+npd$ the number of variables in the extended formulation of the problem.

Let $h_{0}(\theta):= \dsum_{\ell=1}^{m}\lambda_{\ell}\theta_{\ell}$, and denote $\xi_j:=\lceil (\mathrm{deg}\, g_j)/2\rceil$ and $\nu_{j}:=\lceil(\mathrm{deg}%
\,h_{j})/2\rceil$, where $\{g_1,\ldots,g_{n_\mathbf{K}}\}$, and $\{h_1, \ldots, h_{nc1}\}$ are, respectively,  the polynomial constraints that define $\mathbf{K}$ and $\mathbf{\hat K}\setminus \mathbf{K}$ in \ref{pro:locgen2}. For $r\geq r_{0}:=%
\displaystyle\max \{\max_{k=1,\ldots,n_\mathbf{K}} \xi_k,$   $\displaystyle \max_{j=0,\ldots ,nc1}\nu_{j}\}$, introduce the hierarchy of semidefinite
programs:
\begin{equation*}\label{lower2}\tag{$\mathbf{Q1}_{r}$}
\begin{array}{lll}
\displaystyle\min_{\mathbf{y}} & \L_{\mathbf{y}}(p_{\lambda}) &  \\
\mathrm{s.t.} & \Mo_{r}(\mathbf{y}) & \succeq 0, \\
& \Mo_{r-\xi_k}(g_k,\mathbf{y}) & \succeq 0,\quad k=1,\ldots ,n_\mathbf{K}, \\
& \Mo_{r-\nu_{j}}(h_{j},\mathbf{y}) & \succeq 0,\quad j=1,\ldots ,nc1, \\
& y_0= 1,
\end{array}
\end{equation*}%
with optimal value denoted $\inf \mathbf{Q1}_{r}$ (and $\min \mathbf{Q1}_{r}$ if the infimum is attained). Next, based on proposition \ref{pr:lasserreconv} and \cite[Theorem 6.1]{lasserrebook} we can state the following result.

\begin{theorem}
Let $\mathbf{\hat K}\subset\mathbb{R}^{nv_1}$
(compact) be the feasible domain of Problem \ref{pro:locgen2}. Let $\inf \mathbf{Q1}_{r}$ be the optimal value of the semidefinite program \ref{lower2}.
Then, with the notation above:

\textrm{(a)} $\inf\mathbf{Q1}_r\uparrow \rho_{\lambda}$ as $r\to\infty$.

\textrm{(b)} Let $\mathbf{y}^r$ be an optimal solution of the SDP relaxation
 \ref{lower2}. If
\begin{equation*}  \label{finiteconv}
\mathrm{rank}\,\Mo_r(\mathbf{y}^r)\,=\,\mathrm{rank}\,\Mo_{r-r_0}(\mathbf{y}%
^r)\,=\,\varphi
\end{equation*}
then $\min\mathbf{Q1}_r=\rho_{\lambda}$ and one may extract $\varphi$ points $(x_1^*(k),\ldots,x_p^*(k),z^*(k),u^*(k),v^*(k),\zeta^*(k),w^*(k),t^*(k),\\\theta^*(k))_{k=1}^{\varphi}\subset\mathbf{\hat K}$, all global minimizers of the \ref{pro:locgen2} problem.
\end{theorem}
\label{}

\section{Computational Experiments \label{sec:tests}}

We have performed a series of computational experiments to show the efficiency of all the proposed formulations and approaches.
The (mixed integer) SOCP formulations have been coded in Gurobi 5.6 and executed in a PC with an Intel Core i7 processor at 2x 2.40 GHz
 and 4 GB of RAM. We have applied our formulations to the well-known $50$-points data set in Eilon et. al \cite{eilon-watson}
 by considering different number of facilities, different norms and weights. In particular   we have considered the number of  facilities to be located, $p$, ranging in $\{2,5, 10, 15, 30\}$ and $\tau$ (the norm) in $\{\frac{3}{2}, 2, 3\}$. We have run both the non-interchangeable \ref{pro:Multi1} model and the single-allocation model \ref{pro:locgen2}. For the non-interchangeable model we have considered random $\lambda$ and $\mu$ weights (fulfilling the conditions described in (\ref{lambdas})). Table \ref{tableNINT} reports our results on these experiments.  There, we report the average CPU times of $5$ different random instances for each problem (when fixing $p$ and $\tau$).
~\\

\tikzset{mark options={mark size=1.25pt}}

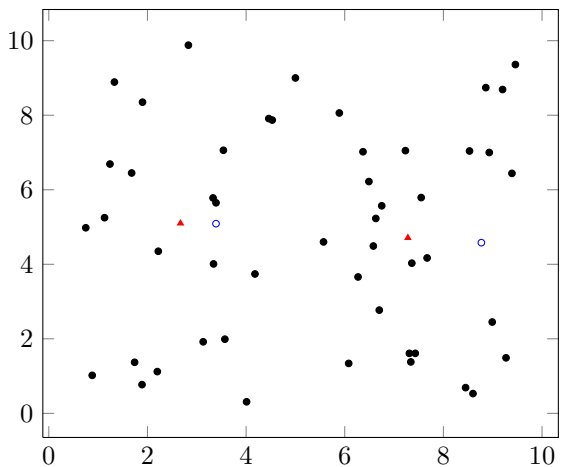
\begin{figure}
\begin{center}
\begin{tikzpicture}

\begin{axis}

\addplot[
scatter,
only marks,
point meta=explicit symbolic,
scatter/classes={
a={mark=*,black},%
w={mark=triangle*,red},%
c={mark=o,blue}},
]
table[meta=label] {
x   y   label
9.46	9.36	a
9.39	6.44	a
9.27	1.49	a
9.20	8.69	a
8.99	2.45	a
8.93	7.00	a
8.86	8.74	a
8.60	0.53	a
8.53	7.04	a
8.45	0.69	a
7.67	4.17	a
7.55	5.79	a
7.43	1.61	a
7.36	4.03	a
7.34	1.38	a
7.31	1.61	a
7.23	7.05	a
6.75	5.57	a
6.70	2.77	a
6.63	5.23	a
6.58	4.49	a
6.49	6.22	a
6.37	7.02	a
6.27	3.66	a
6.08	1.34	a
5.89	8.06	a
5.57	4.60	a
5.00	9.00	a
4.53	7.87	a
4.46	7.91	a
4.18	3.74	a
4.01	0.31	a
3.57	1.99	a
3.54	7.06	a
3.39	5.65	a
3.34	4.01	a
3.33	5.78	a
3.13	1.92	a
2.83	9.88	a
2.22	4.35	a
2.20	1.12	a
1.90	8.35	a
1.89	0.77	a
1.74	1.37	a
1.68	6.45	a
1.33	8.89	a
1.24	6.69	a
1.13	5.25	a
0.88	1.02	a
0.75	4.98	a
7.28    4.71    w
2.67    5.10    w
8.77    4.58    c
3.39    5.09    c
};                  \end{axis}
\end{tikzpicture}
\caption{Points from \cite{eilon-watson} (filled circles) and solutions of $2$-median problem (triangles) and $2$-center problem (empty circles) \label{fi:2med2cen}}.
\end{center}
\end{figure}

\begin{table}[htbp]
  \centering
  \caption{CPU Running times for Non-interchangeable multifacility problem for the Eilon-Watson-Christofides $50$-point data set.}
    \begin{tabular}{c|ccc}
   \backslashbox{$p$}{$\tau$}  & 1.5   & 2     & 3 \\\hline
    2     & 2.5095 & 2.1157 & 3.7470 \\
    5     & 12.7794 & 6.5161 & 9.8130 \\
    10    & 29.1873 & 10.5726 & 19.5455 \\
    15    & 49.4854 & 19.1129 & 40.4506 \\
    30    & 148.7449 & 40.5635 & 85.5676 \\
    \end{tabular}%
  \label{tableNINT}%
\end{table}%
~\\
For general single-allocation multifacility problems, we have consider three of the classical multifacility models that fit the ordered median formulation when particular $\lambda$ weights are chosen: $p$-median problem, $p$-center problem and $p$-$k$-centrum problem (with $k=25$).  In Table \ref{compresults} we report the overall CPU times needed to solve the problems and the optimal solutions provided by Gurobi ($f^*$).  Figure  \ref{fi:2med2cen} depicts the solutions of the $2$-median and $2$-center problems based on the $50$-points data set in Eilon et. al \cite{eilon-watson}.

We observe that the bottleneck for solving \eqref{pro:locgen2} is, apart from its second order cone constraints, the existence of integer variables ($z$ and $w$). In order to ease the resolution some improvements can done via some preprocessing and fixing binary variables based on a geometric  branch and bound phase. The large number of nodes  of the search tree, in some cases, made  that optimality could not be ensured in some instances. This preprocessing is effective  whenever the dimension of the space of variables in low, namely $d=2,3$.

\begin{landscape}

\begin{table}
\label{compresults}
\begin{center}
\begin{tabular}{|r|r|r|r|r|r|r|r|r|r|r|r|r|r|r|}
\hline
\texttt{Problem} & $p$     & $\tau$   & \texttt{CPUTime} & $f^*$     & \texttt{Problem} & $p$     & $\tau$   & \texttt{CPUTime} & $f^*$     &\texttt{Problem} & $p$     & $\tau$   & \texttt{CPUTime} & $f^*$ \bigstrut\\
\hline
\multicolumn{1}{|c|}{\multirow{15}[30]{*}{\rotatebox[origin=c]{90}{$p$-median}}} & \multirow{3}{*}{2}    & 1.5   & 22.31 & 150.955 & \multicolumn{1}{c|}{\multirow{15}[30]{*}{\rotatebox[origin=c]{90}{$p$-center}}} & \multirow{3}{*}{2}     & 1.5   & 1.03  & 4.9452 & \multicolumn{1}{c|}{\multirow{15}[30]{*}{\rotatebox[origin=c]{90}{$p$-25-centrum}}} & \multirow{3}{*}{2}     & 1.5   & 10.08 & 100.8474 \bigstrut\\
\cline{3-5}\cline{8-10}\cline{13-15}\multicolumn{1}{|c|}{} &      & 2     & 1.13  & 135.5222 & \multicolumn{1}{c|}{} &      & 2     & 0.28  & 4.8209 & \multicolumn{1}{c|}{} &     & 2     & 0.38 & 95.0892 \bigstrut\\
\cline{3-5}\cline{8-10}\cline{13-15}\multicolumn{1}{|c|}{} &      & 3     & 23.68 & 130.856 & \multicolumn{1}{c|}{} &      & 3     & 13.51 & 4.788 & \multicolumn{1}{c|}{} &     & 3     & 139.03 & 89.0238 \bigstrut\\
\cline{2-5}\cline{7-10}\cline{12-15}\multicolumn{1}{|c|}{} & \multirow{3}{*}{5}     & 1.5   & 55.28 & 78.6074 & \multicolumn{1}{c|}{} & \multirow{3}{*}{5}     & 1.5   & 3.73  & 2.8831 & \multicolumn{1}{c|}{} & \multirow{3}{*}{5}     & 1.5   & 33.09 & 53.4995 \bigstrut\\
\cline{3-5}\cline{8-10}\cline{13-15}\multicolumn{1}{|c|}{} &      & 2     & 12.49 & 72.2369 & \multicolumn{1}{c|}{} &     & 2     & 5.37  & 2.661 & \multicolumn{1}{c|}{} &      & 2     & 7.61 & 49.6932 \bigstrut\\
\cline{3-5}\cline{8-10}\cline{13-15}\multicolumn{1}{|c|}{} &      & 3     & 125.1 & 68.1791 & \multicolumn{1}{c|}{} &      & 3     & 2.87  & 2.5094 & \multicolumn{1}{c|}{} &      & 3     & 18.23 & 46.9844 \bigstrut\\
\cline{2-5}\cline{7-10}\cline{12-15}\multicolumn{1}{|c|}{} & \multirow{3}{*}{10}    & 1.5   & 5.36  & 45.0525 & \multicolumn{1}{c|}{} & \multirow{3}{*}{10}     & 1.5   & 2.66  & 1.6929 & \multicolumn{1}{c|}{} & \multirow{3}{*}{10}     & 1.5   & 68.36 & 30.7137 \bigstrut\\
\cline{3-5}\cline{8-10}\cline{13-15}\multicolumn{1}{|c|}{} &     & 2     & 2.31  & 41.6851 & \multicolumn{1}{c|}{} &     & 2     & 5.3   & 1.6113 & \multicolumn{1}{c|}{} &     & 2     & 17.93 & 28.9017 \bigstrut\\
\cline{3-5}\cline{8-10}\cline{13-15}\multicolumn{1}{|c|}{} &     & 3     & 4.76  & 39.7222 & \multicolumn{1}{c|}{} &     & 3     & 55.76 & 1.595 & \multicolumn{1}{c|}{} &     & 3     & 225.64 & 27.5376 \bigstrut\\
\cline{2-5}\cline{7-10}\cline{12-15}\multicolumn{1}{|c|}{} & \multirow{3}{*}{15}     & 1.5   & 6.7   & 30.0543 & \multicolumn{1}{c|}{} & \multirow{3}{*}{15}    & 1.5   & 9.44  & 1.1139 & \multicolumn{1}{c|}{} & \multirow{3}{*}{15}    & 1.5   & 49.92 & 22.4165 \bigstrut\\
\cline{3-5}\cline{8-10}\cline{13-15}\multicolumn{1}{|c|}{} &     & 2     & 43.91 & 27.6282 & \multicolumn{1}{c|}{} &     & 2     & 0.62  & 1.0717 & \multicolumn{1}{c|}{} &     & 2     & 11.26 & 20.6536 \bigstrut\\
\cline{3-5}\cline{8-10}\cline{13-15}\multicolumn{1}{|c|}{} &    & 3     & 150.99 & 26.6047 & \multicolumn{1}{c|}{} &     & 3     & 50.08 & 1.053 & \multicolumn{1}{c|}{} &     & 3     & 244.59 & 20.8544 \bigstrut\\
\cline{2-5}\cline{7-10}\cline{12-15}\multicolumn{1}{|c|}{} & \multirow{3}{*}{30}    & 1.5   & 14.45 & 9.9488 & \multicolumn{1}{c|}{} & \multirow{3}{*}{30}    & 1.5   & 74.43 & 1.008 & \multicolumn{1}{c|}{} & \multirow{3}{*}{30}    & 1.5   & 202.54 & 9.0806 \bigstrut\\
\cline{3-5}\cline{8-10}\cline{13-15}\multicolumn{1}{|c|}{} &     & 2     & 4.81  & 8.7963 & \multicolumn{1}{c|}{} &     & 2     & 1.53  & 0.9192 & \multicolumn{1}{c|}{} &     & 2     & 5.29 & 8.521662 \bigstrut\\
\cline{3-5}\cline{8-10}\cline{13-15}\multicolumn{1}{|c|}{} &     & 3     & 198.78 & 8.6995 & \multicolumn{1}{c|}{} &    & 3     & 57.37 & 0.8508 & \multicolumn{1}{c|}{} &    & 3     & 287.90 & 8.001695 \bigstrut\\
\hline
\end{tabular}%
\end{center}
\caption{Computational results for $p$-median, $p$-center and $p$-25-centrum problems for the $50$-points data set in \cite{eilon-watson}.}
\end{table}
\end{landscape}
\section{Reduction of dimensionality of the SDP relaxations for multifacility location problems}
One of the drawbacks of using hierarchies of SDP problems to approximate the solution of our original multifacility location problem is the dimension of the SDP objects that must be used when the relaxation order increases. The size of the matrices to be considered in the SDP problems that have to be solved grows exponentially with the relaxation order. For this reason, it is important to find low rank solutions, that is to identify properties of solutions that ensure significant reduction in the sizes of the problems to be solved. In the following, we address two types of properties that ensure such dimensionality reduction: sparsity and symmetry.

\subsection{Sparsity\label{sec:sparse}}

At times the number of variables that appear in  the polynomial constraints of a polynomial optimization problem can be separated in blocks so that only some of them appear together in some constraints. In those cases  the SDP variables to be used can be simplified  and thus, the sizes of the moment and localizing matrices are dramatically reduced.

The application of this result requires the so call running intersection property \cite{sparse}.

Let $\mathbf{y}=(y_\alpha)$ be a real sequence indexed in the monomial basis $\Upsilon^\alpha=(x^\beta z^\eta     u ^\gamma v^\delta \zeta^\kappa w^\psi t^\tau, \theta^\vartheta)$ of $\mathbb{R}[x,z,u,v,\zeta, w,t, \theta]$ (with $%
\alpha=(\beta, \eta, \gamma, \delta, \kappa, \psi, \tau,\vartheta)\in\mathbb{N}^{pd}\times\N^{np}\times\N^{np}\times\N^{npd}\times\N^{npd}\times\N^{n^2} \times \N^n \times \N^n$).

Let $h_{0}(\theta):=\dsum_{\ell =1}^n \lambda_{\ell} \theta_{\ell}$, and denote $\xi_j:=\lceil (\mathrm{deg}\, g_j)/2\rceil$, $j=1,\ldots,n_\K$, $\nu^{j}_\ell:=\lceil(\mathrm{deg}%
\,h^{j})/2\rceil$, $j=0,\ldots,10$;  where $\{g_1,\ldots,g_{n_\K}\}$ are the polynomial constraints that define $\mathbf{K}$ and $\{h^1, \ldots, h^{10}\}$   are, respectively,   the polynomial constraints  (\ref{c:ti2}) - (\ref{c:n1-7}) in \ref{pro:locgen2}.

Let us denote by $I^x=\{(j,k):j=1,\dots,p,\; k=1,\ldots,d\}$, $I^z=\{(i,j): i=1,\ldots,n,\; j=1,\ldots,p\}$, $I^t=\{i: i=1,\ldots,n\}$, $I^u=\{(i,j): i=1,\ldots,n,\; j=1,\ldots,p\}$, $I^v=I^{\zeta}=\{(i,j,k): i=1,\ldots,n,\; j=1,\dots,p,\; k=1,\ldots,d\}$ and $I^w=\{ (i,\ell): i=1,\ldots,n,\; i=1,\ldots,n\}$. With the above notation we can represent the index multiset that defines the set of variables of Problem  \ref{pro:locgen2} as $I=I^x\cup I^z\cup I^u \cup I^v \cup I^\zeta \cup I^w \cup I^t \cup I^\theta$. Next, for a given$j$, we shall refer to
$I^x(j)=\{(j,k):  k=1,\ldots,d\}$,  $I^z(j)=I^u(j)=\{(i',j):i'=1,\ldots,n \}$, $I^v(j)=I^\zeta(j)=\{(i',j,k):  i'=1,\ldots,n,\; k=1,\ldots,d\}$;  and for a position $\ell=1,\ldots,n-1$,  $I^w(\ell)=\{(i',\ell): i'=1,\ldots,n\}$ and $I^\theta(\ell)= \{\ell, \ell+1\}$.

Now, with $x(I^x)$, $z(I^z(j))$,  $u(I^u(j))$, $v(I^v(j))$, $\zeta(I^\zeta(j))$, $w(I^w)$, $t(I^t)$ and $\theta(I^\theta(\ell))$ we refer, respectively, to the monomials $x$, $z$,  $u$, $v$, $\zeta$, $w$, $t$ and $\theta$  indexed only by subsets of elements in the former sets of indices. Analogously, by $\Upsilon (\tilde I)$, we refer to the monomials extracted from $\Upsilon$ only indexed by the elements in the set $\tilde I$.  (Recall that $\Upsilon=(x,z,u,v,\zeta, w,t, \theta)$ is the set of indeterminates, as defined in Section \ref{s:locomf}.)

Then, for $g_k$, with $k=1,\ldots,n_\K$, let $\Mo_r(\y;I^x)$ (respectively $\Mo_r(g_{k}\y;I^x)$) be the moment (resp. localizing) submatrix obtained from $\Mo_r(\y)$ (resp. $\Mo_r(g_k\y)$) retaining only those rows and columns indexed in the canonical basis of $\mathbb{R}[x(I^x)]$ (resp. $\mathbb{R}[x(I^x)]$). Analogously, for $h^s$, $j=1,\ldots,10$ as defined in (\ref{c:ti2}) - (\ref{c:n1-7}) , respectively, let $\Mo_r(\y;\tilde I)$ (respectively $\Mo_r(h^{j}\y;\tilde I)$,  be the moment (resp. localizing) submatrix obtained from $\Mo_r(\y)$ (resp. $\Mo_r(h^j\y)$) retaining only those rows and columns indexed in the canonical basis of
$\mathbb{R}[\Upsilon(\tilde I))]$.

Let $\tilde{I}(0)=I^x\cup I^\theta \cup  I^t$  and
$\tilde{I}(j)=I^x(j)\cup I^z(j)\cup I^u(j)\cup I^v(j) \cup I^\zeta(j)\cup I^t$ for $j=1,\ldots,p$ and $\tilde{I}(p+\ell) = I^w(\ell) \cup I^\theta(\ell)$.

Observe that
\begin{equation}
\tilde{I}(j+1) \cap \bigcup_{j'\le j}  \tilde{I}(j')\subseteq \tilde{ I}(0), \quad \forall j=1, \ldots, p+n-1.
\end{equation}

The above subdivision of variables allows us to partition the polynomial constraints of \ref{pro:locgen2} (except those representing binary variables)  into $ p + n+1$ groups, ${\rm F}(j)$ and ${\rm F}(p+\ell)$, one for each $\tilde{I}(j)$ and each $\tilde{I}(p+\ell)$ for $j=0, \ldots, p$ and $\ell=1, \ldots, n$  so that within each group the constraints only depend of the variables in $\tilde{I}(j)$ for $j=0,\ldots,p$ and $\tilde{I}(p+\ell)$ for $\ell=1,\ldots,n$. Specifically, let
$$ F(0)=\{g_r(x)\ge 0: r=1,\ldots,n_{\K}\},$$
and for fixed  $j=1, \ldots, p$:
\begin{equation} \hspace*{-2cm} (F(j))  \left\{ \begin{array}{rll}
 h_{j}^3&:=u_{ij}\le t_i+UB_i(1-z_{ij}), &  i=1,\ldots, n,\\
 h_{j}^4&:=v_{ijk}-x_{jk}+a_{ik}\ge 0,& i=1,\ldots,n, k=1, \ldots, d,\\
 h_{j}^5&:=v_{ijk}+x_{jk}-a_{ik}\ge 0,& i=1,\ldots,n,\; k=1,..,d,    \\
 h_{j}^6&:=v_{ijk}^r\leq \zeta_{ijk}^{s}u_{ij}^{r-s},& i=1,\ldots,n,\; k=1,\ldots,d,  \\
 h_{j}^7&:=\sum_{k=1}^d \zeta_{ijk}\le u_{ij},& i=1,\ldots,n. \\
\end{array} \right.  \label{ijl-region}
\end{equation}
Next, for fixed  $\ell=1, \ldots, n-1$:
\begin{equation}
\hspace*{-3cm} (F(p+\ell))  \left\{\begin{array}{rll}
h_{\ell}^1&:=t_i\le \theta_{\ell} +UB_i(1-w_{i\ell}), & i=1,\ldots,n, \\
 h_{\ell}^2&:=\theta_{\ell}\ge \theta_{\ell+1}, &
\end{array} \right.  \label{ijl-region2}
\end{equation}
and
\begin{equation}
\hspace*{-3cm} (F(p+n)) \qquad  h^1_n:= t_i\le \theta_n+UB_i(1-w_{in}), \quad i=1,\ldots,n. \label{ijl-region3}
\end{equation}

For $r\geq \max\{ r_{0},\nu_0\}$ where
$r_0:=\displaystyle \max_{k=1,\ldots,\ell} \xi_k$  and $\nu_0:=\displaystyle \ \max_{j=0,\ldots ,10}\nu^{j}_\ell$, we introduce the following hierarchy of semidefinite programs:
\begin{equation*}\label{lower0}\tag{$\mathbf{Q1}_{r}^{\rm sp}$}
{\small \begin{array}{lll}
\displaystyle\inf_{\mathbf{y}} & \L_{\y}(\dsum_{\ell=1}^n\lambda_{\ell}\theta_{\ell}) &  \\
\mathrm{s.t.} & \Mo_{r}(\y;\tilde{I}(0)) & \succeq 0, \\
& \Mo_{r-\xi_k}(g_k\y;\tilde{I}(0)) & \succeq 0,\; k=1,\ldots ,n_\K, \\
& \Mo_r(\y;\tilde{I}(j))& \succeq 0,\;  j=1,\dots,p\\
& \Mo_r(\y;\tilde{I}(p+\ell))& \succeq 0,\; \ell=1,\ldots ,n-1,\\
& \Mo_{r-\nu_{\ell} ^1}(h_j^{1}\y ;\tilde{I}(p+\ell))& \succeq 0,\; \ell=1,\dots,n,\\
& \Mo_{r-\nu_{\ell} ^2}(h_j^{2}\y ;\tilde{I}(p+\ell))& \succeq 0,\; \ell=1,\dots,n-1,\\
& \Mo_{r-\nu_{j} ^3}(h_j^{3}\y ;\tilde{I}(j))& \succeq 0,\; j=1,\dots,p,\\
& \Mo_{r-\nu_{j} ^4}(h_j^{4}\y ;\tilde{I}(j))& \succeq 0,\; j=1,\dots,p,\\
& \Mo_{r-\nu_{j} ^5}(h_j^{5}\y ;\tilde{I}(j))& \succeq 0,\; j=1,\dots,p,\\
& \Mo_{r-\nu_{j} ^6}(h_j^{6}\y ;\tilde{I}(j))& \succeq 0,\; j=1,\dots,p,\\
& \Mo_{r-\nu_{j} ^7}(h_j^{7}\y ;\tilde{I}(j))& \succeq 0,\; j=1,\dots,p,\\
& \L_y(\dsum_{j=1}^p z_{ij}-1)& = 0, \;   i=1,\dots,n,\\
& \L_y(\dsum_{i=1}^n w_{i\ell}-1)& = 0, \;   \ell=1,\dots,n,\\
& \L_y(\dsum_{\ell=1}^n w_{i\ell}-1)& = 0, \; i=1,\ldots,n,\\
& \L_y(w_{i\ell}^2-w_{i\ell})& =0, \; i,\ell=1,\ldots,n, \\
& \L_y(z_{ij}^2-z_{ij})& =0, \; i=1,\ldots,n,\; j=1,\dots,p,\\
\end{array}%
}
\end{equation*}%
with optimal value denoted $\inf \mathbf{Q1}^{\rm sp}_{r}$.

\begin{theorem} Let $\mathbf{\overline K}\subset\mathbb{R}^{nv_1}$  be the feasible domain of \ref{pro:locgen2}.
Then, with the notation above:

\textrm{(a)} $\inf \mathbf{Q1}^{\rm sp}_{r}\uparrow \rho_{\lambda}$ as $r\to\infty$.

\textrm{(b)} Let $\mathbf{y}^r,$ be an optimal solution of the SDP relaxation \ref{lower0}.  If
\begin{eqnarray}  \label{finiteconv1}
\mathrm{rank}\,\Mo_r(\mathbf{y}^r;I^x \cap I^t \cap I^\theta)&= & \mathrm{rank}\,\Mo_{r-r_0}(\mathbf{y}%
^r; I^x \cap I^t \cap I^\theta)\nonumber\\ {\small \hspace*{-0.75cm}
\mathrm{rank}\,\Mo_r(\mathbf{y}^r;\tilde{I}(j))} & = & {\small \mathrm{rank}\,\Mo_{r-\nu_0}(\mathbf{y}%
^r;\tilde{I}(j)) \; j=1,\ldots,p \label{finiteconv2}}
\\ {\small \hspace*{-0.75cm}
\mathrm{rank}\,\Mo_r(\mathbf{y}^r;\tilde{I}(p+\ell))} & = & {\small \mathrm{rank}\,\Mo_{r-\nu_0}(\mathbf{y}%
^r;\tilde{I}(p+\ell)) \; \ell =1,\ldots,n\label{finiteconv3}}
\end{eqnarray}
and if $\mathrm{rank}(\Mo_r(\y^r; I^t))= \mathrm{rank}(\Mo_r(\y^r; I^\theta (\ell) \cap I^\theta(\ell+1))) = \mathrm{rank}(\Mo_r(\y^r; I^t \cap I^x(j))) = \mathrm{rank}(\Mo_r(\y^r; I^\theta(\ell)))= 1$ for all $j=1, \ldots, p$, $\ell = 1, \ldots, n-1$.

Moreover, let $\Delta_{j,\ell}:=\{(x(I^x(j))^*,z(I^z(j))^*, u(I^u(j))^* ,v(I^v(j))^*, \zeta(I^\zeta(j))^*, w(I^w(\ell))^*, t(I^t)^*, \theta(I^\theta(\ell))^*\}$ be the set of solutions obtained by the application of the condition (\ref{finiteconv2}) and (\ref{finiteconv3}). Then, every $(x^*,z^*,u^*,v^*,\zeta^*,w^*,t^*,\theta^*)$ such that $(x^*_{jk},z^*_{ij},u^*_{ij},v^*_{ijk}, \zeta_{ijk}^*, w_{i\ell}, t_i^*, \theta^*_\ell)$ has indices in $\tilde{I}(j)$, for $j=1, \ldots, n+p$, is an optimal solution of Problem $\mrf_{\lambda}$.
\end{theorem}
\begin{proof}
The convergence of the semidefinite relaxation \ref{lower0} was proved by Jibetean and De Klerk \cite[Theorem 9]{jibetean} for a general rational function over a basic, closed semi-algebraic set that satisfies Archimedean Property. Here, we use that result applied to Problem  \ref{pro:locgen2}. Moreover, the index set of the indeterminates in the feasible set that generate localizing constraints, namely constraints (\ref{c:ti2}) - (\ref{c:n1-7}) admits the decomposition  $\tilde{I}(0), \tilde{I}(1), \ldots, \tilde{I}(p+n)$  that satisfies the running intersection property \cite{sparse}.

Indeed,
the index sets $I=\{1, \ldots, nv_1\}$ and $J=\{1, \ldots, n_K+nc_1\}$ are partitioned into  sets $\{\tilde{I}(j)\}_{j=0}^p\cup \{\tilde{I}(p+\ell)\}_{\ell=1}^n$ and $\{F(j)\}_{j=0}^p\cup \{F(p+\ell)\}_{\ell=1}^n$, respectively, satisfying:
\begin{enumerate}
\item $\{F(\cdot)\}$ are disjoint sets for all $j$ and $\ell$.
\item For every $j =0,\ldots,p$, the constraints in the set $F(j)$ only involve variables in the set $\tilde{I}(j)$ and for every $\ell =1,\ldots,n$, the constraints in the set $F(p+\ell)$ only involve variables in the set $\tilde{I}(p+\ell)$.
\item The objective function $f$ can be written as $f=\dsum_{\ell=1}^n \lambda_{\ell} \theta_{\ell}$ where $\lambda_{\ell} \theta_{\ell} \in \R[\Upsilon(\tilde{I}(p+\ell)]$ for $\ell=1, \ldots, n$.
\item  The decomposition of the index set of variables satisfies:
\begin{itemize}
\item $\tilde{I}(j) \cap \tilde{I}(j')= I^t$ for $j, j'=1, \ldots, p$ with $j\neq j'$.
\item $\tilde{I}(j) \cap \tilde{I}(p+\ell)= \emptyset$ for $j,1, \ldots, p$, $\ell=1, \ldots, n$.
\item $\tilde{I}(p+\ell) \cap \tilde{I}(p+\ell')= \emptyset$ for $\ell, \ell'=1, \ldots, p$ with $\ell \neq \ell'-1, \ell', \ell'+1$.
\item $\tilde{I}(p+\ell) \cap \tilde{I}(p+\ell+1)= I^\theta (\ell) \cap I^\theta(\ell+1)$ for $\ell=1, \ldots, n-1$.
\item $\tilde{I}(0) \cap \tilde{I}(j)= I^t \cup I^x(j)$ for $j=1, \ldots, p$.
\item $\tilde{I}(0) \cap \tilde{I}(p+\ell)= I^\theta(\ell)$ for $\ell=1, \ldots, n-1$.
\end{itemize}
Hence, it is clear that:
$ \tilde{I}(j+1) \cap \bigcup_{j'\le j}  \tilde{I}(j')\subseteq \tilde{ I}(0), \quad \forall j=1, \ldots, p+n-1.$
\end{enumerate}

Therefore, the result follows by combining  \cite[Theorem 3.2]{sparse} and  \cite[Theorem 9]{jibetean} .

\par \fin

\end{proof}

The above theorem allows us to approximate and solve the original problem $\mrf _{\lambda}^0$ by its relaxation  \ref{lower0} up to any degree of accuracy by solving block diagonal (sparse) SDP programs which are convex programs for each fixed relaxation order $r$ and that can be solved, up to any given accuracy, in polynomial time with available open source solvers as SeDuMi, SDPA, SDPT3, etc. Observe that the reduction in the number of variables of the semidefinite programs is remarkable. For instance, for the relaxation of order $r$  from $O(nv_1^{2r})$ to $O((p+n+1)\left(\frac{nv_1}{p+n+1}\right)^{2r})$.

\subsection{Symmetry}

In this section, we describe how to exploit the symmetry  of Problem \ref{pro:locgen2} to construct another much simpler SDP relaxation. This relaxation is based on the invariance of this problem under the action of the symmetric group applied to the indices of the facilities. The general framework and theoretical details about the reduction of dimensionality of polynomial optimization problems based on the invariance under general compact groups are given in \cite{lasserre-sym}.

 In this section we are interested in applying particular reductions to Problem \ref{pro:locgen2} based on the invariance of such a problem under the symmetric group. We denote by $\sym_p$ the symmetric group on $p$ variables that represent the $p$ facilities that want to be located. Note that the $j$ label established in our formulation for the facilities is arbitrary and that several optimal solutions appear when the facilities are re-labeled using a different number for them. Hence, if $x_1^*, \ldots, x_p^*$ is a set of optimal facilities for our problem, $x_{\sigma(1)}^*, \ldots, x_{\sigma(p)}^*$ is also a set of optimal facilities for any $\sigma \in \sym_p$. In what follows we want to use this symmetry of the decision variables to reduce the sizes of the matrices involved in the hierarchy of semidefinite programs that converge to the optimal solution when the moment approach is applied.

Recall that the complete set of variables of \ref{pro:locgen2} is $(x,z,u,v,\zeta, w,t, \theta)$. We also denote by $N =d+2n+2nd$ and $M=n(n+2)$. Note that the overall number of variables is $nv_1=Np+M$.

We will apply symmetry results when permuting the $j$-indices in the set of variables $\Upsilon=\{x, z, u, v, \zeta\}$ and keeping $w$, $t$ and $\theta$ invariant. Hence, we consider the following  action $\varphi$ over $\R^{p}$:
$$
\varphi: \sym_p \times \R^{p} \rightarrow \R^{p}
$$
defined as  $\varphi(\sigma, (y_1, \ldots, y_p)) = (y_{\sigma(1)}, \ldots, y_{\sigma(p)})$ for any $\sigma \in \sym_p$ and $y \in \R^p$.

With this action, we can define the following action over the overall set of variables of our problem:

$$\varphi_\Upsilon: \sym_p \times \R^{Np+M} \rightarrow \R^{Np+M}$$

\noindent defined such that $\varphi_\Upsilon$ maps $(\sigma, (x,z,u,v,\zeta,w,t,\theta))$ into $(\varphi(\sigma, x(I^x(1))), \ldots, \varphi(\sigma, v(I^v(i,k))), \varphi(\sigma, \zeta(I^\zeta(i,k))),$ \\
$ w(I^w), t(I^t),\theta(I^\theta))$,  i.e., permuting the indices associated with facilities in the decision variables (the $j$-index).

Hence, from the definition of $\varphi$, and denoting for each pair of actions $\varphi$ and $\psi$ of the symmetric group $\sym_p$ over $\R^p$, $\varphi \oplus \psi: \sym_p \times \R^{p+p} \rightarrow \R^{p+p}$ as $\varphi \oplus \psi (\sigma, (y, y')) = (\varphi(\sigma, y), \psi(\sigma, y'))$, we get the following result:

\begin{lemma}
\label{lem:dec0}
$\varphi_\Upsilon = \varphi \oplus \stackrel{N}{\cdots} \oplus \varphi \oplus 1_M$.
\end{lemma}

 We say that a polynomial $p(x)$ is invariant under an action $\varphi: \sym_p \times \R^n \rightarrow \R^n$ if $p(\varphi(\sigma,x))=p(x)$, for all $\sigma \in G$. A polynomial optimization problem is said invariant under $\varphi$ if both the objective function and the constraints are invariant under the action. Hence, it is straightforward to check that the objective function and constraints that define Problem \ref{pro:locgen2} are invariant under the  action of $\varphi_\Upsilon$ over the set of variables.

\begin{lemma}
Let $\sym_p$ be the symmetric group on the set $\{1, \ldots, p\}$. Then, \eqref{pro:locgen2}  is invariant under $\varphi_\Upsilon$.
\end{lemma}

Furthermore, in practice we are interested in applying not only the symmetry results but also the sparsity results given in Section \ref{sec:sparse}. In this point, note that the sparsity results are not affected by the symmetry reduction presented in this section since the  action $\varphi_\Upsilon$ maps the blocks  $F(0)$, $F(p+\ell)$ for all $\ell=1,\ldots,n$  into themselves and the block ${\rm F}(j)$  for $j=1,\ldots,p$ is permuted with ${\rm F}(\sigma(j), l)$ when a fixed permutation $\sigma$ is given. Hence, both reductions can be applied to our problem.

The results on the reduction of dimensionality by symmetry are based on the Theory of Representations, in particular on the decompositions of the symmetric group $\sym_p$ into irreducible $\sym_p$-modules (the interested reader is refereed to \cite{steimberg} fur further details about the general theory for finite groups). Note that our action is somehow special because of its construction as a direct sum of actions. We show in the following result that once a decomposition into irreducible representations is computed, we can easily compute the decomposition of convenient modifications of the original representation.
\begin{lemma}
\label{lemma:dec}
 Let $p, K,N \in \N$, and $\rho: \sym_p \rightarrow GL(V)$ a representation of $\sym_p$. Consider $\rho = \rho_1 \oplus \cdots \oplus \rho_k$ a decomposition into irreducible representations of $\sym_p$. Then, the representation $\tilde{\rho}=1_K \oplus N \rho: G \rightarrow GL(V^{K+Np})$, $\tilde{\rho}(\sigma) = {\rm Diag}(I_K, \rho(g), \stackrel{N}{\cdots}, \rho(g))$ can be decomposed into irreducible representations as:
$$
\tilde{\rho} = 1_K \oplus N \rho_1 \oplus \cdots \oplus N \rho_k
$$
where $N\rho_i = \rho_i \oplus \stackrel{N}{\cdots} \oplus \rho_i$, for all $i$.
\end{lemma}
\begin{proof}
Clearly, each representation in the decomposition of $\tilde{\rho}$ is irreducible by hypothesis. Furthermore, since $\rho = \rho_1 \oplus \cdots \oplus \rho_k$, $1_K \oplus N \rho_1 \oplus \cdots \oplus N \rho_k = 1_K \oplus N(\rho_1 \oplus \cdots \oplus \rho_k) = 1_K \oplus N\rho = \tilde{\rho}$.
\qed
\end{proof}

By Maschke's Theorem (see \cite[Thm 1.5.3]{sagan-2001}) applied to the symmetric group, every $\sym_p$-module $V$ is a direct sum of irreducible $G$-submodules of $V$, i.e.,
\begin{equation}\label{eq:maschke} V \ \cong \ \bigoplus_{i=1}^{s}V_{i}
\text{\, with irreducible $\sym_p$-submodules } V_{i} \, .
\end{equation}
Each irreducible $\sym_p$-submodule might occur several times in the direct
sum.

Recall that a scalar product $\langle\cdot,\cdot\rangle$ on $\R[X]$ is $\sym_p$-invariant  if
$\langle g,f\rangle =\langle g_\sigma,f_\sigma\rangle $ for every $f,g\in\R[X]$ and every $\sigma\in \sym_p$.

Assume a decomposition of $\R[X]$ like in \eqref{eq:maschke}, consider $V_i = W_{i1} \oplus \cdots \oplus W_{i\eta_i}$ and  pick any $b_{i,1,1}\in W_{i1}$. Then using the fact that the $W_{ij}$ are isomorphic we can find a vector $b_{i,j,1}\in W_{ij}$ such that $\phi_{i,j}(b_{i,j,1})=b_{i,j+1,1}$, where $\phi_{i,j}$ is a $\sym_p$-isomorphism that maps $W_{i,j}$ to $W_{i,j+1}$. Now using for example Gram-Schmidt every $b_{i,j1}$ can be extended to an orthonormal basis of $W_{ij}$. We will call such a  resulting basis $\mathcal{B}=\{b_{1,1,1},\ldots,b_{h,\eta_h,\nu_h}\}$ a \emph{symmetry-adapted basis} of $\R[X]$.

We detail now how to compute the decompositions of representations of the symmetric group, the symmetry-adapted bases and how to apply those decompositions to solve Problem \ref{pro:locgen2}. First of all, we concentrate on the problem of determining irreducible  representations and symmetry adapted bases of $\R[Y]=\R[Y_1, \ldots, Y_p]$ for the symmetric group $\sym_p$. The irreducible representations of $\sym_p$ are in bijection with the partitions of $p$. Thus, in order to construct the suitable moment matrix that split this structure, we need a decomposition of the ring of polynomials into $\sym_p$-irreducible components.  The classical construction of Specht based on the theory of tableaux allows this representation \cite{spetch}.

Let $\lambda \vdash p$ be a partition of $p$, i.e., a sequence of non increasing positive integers $\lambda_1, \ldots, \lambda_k$ with $\sum_{l=1}^k \lambda_l = p$. Given two partitions $\lambda, \mu \vdash p$, we say that $\lambda$ dominates $\mu$, and we write $\lambda \trianglerighteq \mu$, if $\sum_{l=1}^i \lambda_i \geq  \sum_{l=1}^i \mu_i$ for all $i$. A Young tableaux for a partition $\lambda \vdash p$ consists of as many rows as components in $\lambda$, each of them with as many columns as the number in the corresponding component in $\lambda$.  The entries are the elements in $\{1, \ldots, n\}$, and each of these numbers appear exactly once.

A generalized Young tableau of shape $\lambda \vdash p$ and content $\mu \vdash n$ is a Young tableau for $\lambda$ whose entries are $\{1, \stackrel{\mu_1}{\ldots}, 1\}, \ldots, \{\ell, \stackrel{\mu_\ell}{\ldots}, \ell\}$. A generalized Young tableau is called semistandard if its rows are non increasing and its columns are strictly increasing.

Given a partition $\lambda\vdash p$, and a $\lambda$-Young tableau $t$, the class of equivalence of $t$ (by rows) is denoted by $\{t\}$ and is called a tabloid. The permutation module $M^\lambda$ is then the $\sym_p$-module defined by $\C\{\{t_1\}, \ldots, \{t_\ell\}\}$, where $\{\{t_i\}_{i=1}^\ell\}$ is the complete list of $\lambda$-tabloids.

Two tableaux are said \textit{equivalent} if the corresponding rows and columns of the two tableaux contain the same numbers. The equivalence class of a tableau $T$ is denoted by $[T]$.

The column stabilizer of a $\lambda$-tableau $t$ is ${\rm Cstab}_t = \sym_{C_1} \times \cdots \times \sym_{C_\nu}$ where $C_1, \ldots, C_\nu$ are the columns of the tableau and $\sym_{C_i}$ is the symmetric group on $C_i$. The polytabloid associated to a tabloid $T$ is defined by:
$$
e_t = \sum_{\sigma \in {\rm CStab}_T} sgn(\sigma) \sigma(\{T\}).
$$

The irreducible representations of the symmetric group $\sym_p$ are in one-to-one correspondence with the partitions of $p$, and they are given by the Specht modules.
The Specht module of $\lambda \vdash p$, $S^\lambda$ is the submodule of the permutation module $M^\lambda$ spanned by the polytabloids $\e_T$. The dimension of $S^\lambda$, $dim(S^\lambda)$, is the number of standard Young tableaux for $\lambda$.

On the other hand, for $\beta \in \N^p$, let us denote $\R\{Y^\beta\} = \{Y^{\sigma(\beta)}: \sigma \in \sym_p\}$. Then, it is clear that
$$
\R[Y]_{\leq r} = \bigoplus_{s=0}^r \; \bigoplus_{\beta \in \N^{p}_s} \R\{Y^\beta\},
$$
where $\N^{p}_s = \{\beta \in \N^p:  |\beta|=s\}$.

With the above decomposition, it would be enough then to construct a decomposition of $\R\{Y^\beta\}$ into $\sym_p$-irreducible components,

Let $\xi \in \N^p$. We denote by $b^\xi \in \N^q$ the vector of distinct components of $\xi$ ordered (non increasingly) according to the multiplicity of the occurrence in $\xi$. Then, we construct $\mu\in \N^q$ as $\mu_j = \# \{i \in \{1, \ldots, p\}: \xi_i = b^{\xi}_j\}$. For example, for $\xi= (4,2,2, 4, 7, 4) \in \N^6$, we get that $b^\xi=(4,2,7)$ and $\mu=(3,2,1)$ since $4$ appears $3$ times in $\xi$, $2$ appears twice and $7$ once.

We analyze now how to express $\R\{Y\}$ in terms of irreducible $\sym_p$-invariant modules.  Let $b^\beta$ as defined above.

For a partition $\lambda \vdash p$, a semi-standard $\lambda$-tableau $t_\lambda$ and $T$ a  generalized Young tableau with shape $\lambda$ and content $\mu^\beta$ we define:
$$
\mathbf{Y}^{(t_\lambda, T)} =\prod_{k} \prod_{(r,c)} Y_{ t_\lambda(r,c),k}^{b^\beta_{T(r,c)}}
$$
where, $T(r,c)$ is the element in its $c$-th column and $r$-th row.

Next, for each column $c$ of the $\lambda$-tableaux, we construct the Vandermonde determinant as:
$$
Van_c^{T} = \det \left(\begin{matrix}  Y_{ t_\lambda(1,c)}^{b^\beta_{T(1,c)}} & \cdots & Y_{ t_\lambda(q,c)}^{b^\beta_{T(1,c)}}  \\
\vdots & \ddots & \vdots\\
 Y_{ t_\lambda(1,c)}^{b^\beta_{T(q,c)}}  & \cdots & Y_{ t_\lambda(q,c)}^{b^\beta_{T(q,c)}}
 \end{matrix}\right)
 $$

With this determinant, we construct the polynomial:
$$
s_{t_{\lambda}, T} = \prod_c Van_c^{T},
$$
and finally the Specht polynomial:
$$
S_{t_{\lambda},T} = \dsum_{S \in [T]} s_{t_{\lambda}, S}.
$$

For a partition $\mu \vdash p$, $M^\mu$ is the $\sym_p$-module generated by the complete list of $\mu$-tabloids. $S^\mu$ is the submodule of $M^\mu$ generated by the polytabloids $e_T = \sum_{\sigma \in CStab_T} sgn(\sigma) \sigma\{T\}$ for each Young tableau for $\mu$. The following result states that the Specht polynomials generate a submodule of $\R\{Y^\beta\}$ equivalent to the Specht module.

\begin{lemma}
\label{lemma:specht}
 Let $T$ be a generalized Young tableau with shape $\lambda\vdash p$ and content $\mu^\beta$. The generalized Specht polynomials $S_{(t_\lambda, T)}$ generate an $\sym_p$-submodule of $\R\{Y^\beta\}$ which is isomorphic to the Specht module $S^\lambda$.
\end{lemma}

With the above results, we get the following result which is proven in \cite{lasserre-sym}.

\begin{theorem}
Let $\beta \in \N^p_0$ with $\dsum_{i=1}^p \beta_i = r$ and shape $\mu^\beta$. Then:
$$
\R\{Y^\beta\}  = \bigoplus_{\lambda\unrhd\mu^\beta} \bigoplus_{T \in \mathcal{T}_{\lambda,\mu}} \R\{S_{t_\lambda,T}\}
$$
where $t_\lambda$ denotes the unique $\lambda$-tableau with increasing rows and columns and $\mathcal{T}_{\lambda, \mu}$ the set of  semistandard generalized Young  tableaux of shape $\lambda$ and content $\mu$.
\end{theorem}

Once the general theory for decomposing the standard symmetric group is detailed, our action $\varphi$ needs to be treated slightly different since our set of variables has $Np+M$ elements and the symmetric group has $p$. To construct the basis of invariant polynomials  (with degree at most the desired relaxation order) in the complete set of variables of \ref{pro:locgen2}, once we know a basis of invariant polynomials in $p$ variables under $\sym_p$, we reproduce such a basis on each of the set of variables where the action is $\varphi_{(i,j, \ell)}$ is non trivial and completed it with the standard monomial basis where the action is defined as the identity.

\begin{lemma}
\label{lem:bases}
Let $\mathcal{B}_k(Y)$ be a symmetry-adapted basis of $\R[Y_1, \ldots, Y_p]$ of degree at most $k$ and $\mathcal{B}^{st}(X)$ the standard monomial basis of $\R[w(I^w), t(I^t)]$ with degree at most $k$. Then, the elements of a symmetry-adapted basis of $\R[x,z,u,v,\zeta,w,t,\theta]$ are of the form:

$$
b = b^{x_1} \cdots b^{v_{n,d}} \cdot b'
$$
where $b^{x_k} \in \mathcal{B}_k (x(I^x(k)))$, $b^{z_i} \in \mathcal{B}_i (z(I^z(i)))$,  $b^{u_i} \in \mathcal{B}_k (u(I^u(i)))$, $b^{v_{i,k}} \in \mathcal{B}_k (v(I^v(i,k)))$, for $i=1, \ldots, n$, $k=1, \ldots, d$, and $b' \in \mathcal{B}^{st}(X)$ and such that  $deg(b) \leq k$.
\end{lemma}

\begin{proof}
The result follows from the equivalence between $R[(x,z,u,v,\zeta,w,t,\theta)]$ and the construction of such a ring by the extension $
\R[x(I^x(1))]\cdots [\zeta(I^\zeta(n,d))][w(I^w),t(I^t),\theta(I^\theta)]$, where for a given multivariate polynomial ring $K[Y]$, and a set of variables $Y'$, the polynomial ring $K[Y][Y']$ is defined as the set of polynomials in the variables $Y'$ with coefficients in $K[Y]$.  Hence a basis for $R[\Upsilon_{(0,0)}]$ can be constructed by iteratively multiplying the element in bases for the polynomial rings $R[x(I^x(1))], \ldots, \R[\zeta(I^\zeta(n,d))]$ and $\R[w(I^w), t(I^t), \theta(I^\theta)]$.

For instance, let $f  \in \R[(x,z,u,v,\zeta,w,t,\theta)]$. Then, by dividing $f$ by the standard monomial basis with term ordering $\succ$ such that $x_1 \succ \cdots \succ x_d \succ \cdots \succ \zeta_{nd} \succ w \succ t \succ \theta$, it can be written as
$$
f = q_1^{x_1} \cdots q_1^{x_d} q'_1 +  q_2^{x_2} \cdots q_2^{x_d} q'_2 + \cdots + q_d^{x_d}  q'_d  + q'_{0}
$$
where $q_j^{x_k} \in \R[x(I^x(k))]$, for $k=1, \ldots, d$, and $q_i' \in \R[w(I^w), t(I^t),\theta(I^\theta)]$, for $j=0, \ldots, d$. Now applying each of the bases to each polynomial, we get the desired expression.

In the polynomial ring $\R[w(I^w), t(I^t),\theta(I^\theta)]$ no symmetry is applied, so the standard basis is used. For the first $d$ rings, the symmetry adapted basis is used.
\qed
\end{proof}

Once the symmetry-adapted bases of our problem are computed, the symmetry-adapted SDP-relaxation must be stated. Note that, we have a symmetry-adapted basis for each shape $\lambda \vdash n$. Each of those bases will give as a block in the overall moment matrix. In order to do  that, we consider a $\sym_p$-linear map, i.e., a linear map $L^{\rm sym}: \R[\Upsilon] \rightarrow \R$ such that $L^{\rm sym}(f(X)) = L^{\rm sym} (f(\varphi(\sigma,X)))$ for all  $\sigma \in \sym_p$.  We also define the following bilinear form, for any $g \in R[\Upsilon]$:
$$
\mathcal{L}^{\rm sym}_{g}: \R[X] \times \R[X] \longrightarrow \R
$$
$$
(p,q) \mapsto  {L}^{\rm sym} \left(\frac{1}{p} \dsum_{\sigma \in \sym_p} p(\varphi(\sigma, X)) \cdot q(\varphi(\sigma, X)) \cdot g\right)
$$

For the sake of simplicity, we will use the following $\sym_p$-linear map for our computations: for each subset of variables $\{Y_1, \ldots, Y_p\}$ in $\{x(I^x(1)), \ldots, x(I^x(d))), \ldots, v(I^v(i,k)),\zeta(I^\zeta(i,k))\}$ we fix $L^{\rm sym}(Y_1)=y_v \in \R$. Since the mapping must be invariant under $\sym_p$ and there is always a permutation $\sigma \in \sym_p$ that maps $1$ in $\sigma(1) \in \{1, \ldots, p\}$, we know that $L^{\rm sym}(Y_j)=y_v$ for $j=1, \ldots, p$.

With the about settings, we define the symmetry-adapted moment matrices as follows:

\begin{cor}\label{thm:symmom}
For $r\in\N$, the $r$-th symmetry-adapted moment matrix $M_r^{\rm sym}(y)$ is
of the form
$$
M_r^{\rm sym}(y) \ = \ \bigoplus_{\lambda\vdash n} M_{r,\lambda}^{\rm sym}(y),
$$
\end{cor}
where $M_{r,\lambda}^{\rm sym}(y)$ is the moment matrix constructed for symmetry-adapted basis for fixed $\lambda \vdash n$ as in the standard case (see Section \ref{sec:prelim}).

Finally, we define the symmetry-adapted relaxation
\begin{equation}\label{eq:lasserresym}
  Q^{\rm sym}_r:\quad\begin{array}{rcl}
   \multicolumn{3}{l}{\inf_{y} L^{\rm sym}(p) }\\
  M_r^{\rm sym}(y) & \succeq & 0 \, , \\
  M_{r - \lceil \deg g_j / 2 \rceil}^{\rm sym}(g_k \, y) & \succeq & 0 \, , \quad
k \le j \le n_{\K} \,\\
 M_{r - \lceil \deg h_j / 2 \rceil}^{\rm sym}(h_j \, y) & \succeq & 0 \, , \quad
1 \le j \le nc_1 \,
  \end{array}
\end{equation}
with optimal value denoted by $\inf Q^{\rm sym}_r$ (and $\min Q^{\rm sym}_r$ if the infimum
is attained).
\begin{remark}\label{re:size}
The symmetry-adapted setting defined above can give a significant reduction of the SDPs objects that need to be calculated. Indeed the number of variables involved equals the size of $\mathcal{B}_k$. Furthermore the symmetry-adapted moment matrix is block diagonal and the size of each block equals $\eta_i$.
\end{remark}

In this setting, Proposition
\ref{pr:lasserreconv} can be reformulated as follows.
\begin{theorem}\label{thm:symscheme}
Assume that the Archimedean Property \ref{defput} holds and let
$(Q^{\rm sym}_r)_{r\geq k_0}$ be the hierarchy of SDP-relaxations defined in
\eqref{eq:lasserresym}.
Then $(\inf Q^{\rm sym}_r)_{r\geq k_0}$ is a monotone non-decreasing sequence
that converges to $p^*$, i.e., $\inf Q^{\rm sym}_r\uparrow p^*$ as $r\to\infty$.
\end{theorem}

\begin{ex}
Let us consider a toy example for Problem \ref{pro:locgen2} with   $n=3$  demand points, in dimension $d=2$, $p=2$ facilities to be located and relaxation order $k=2$. The set of variables of our problem is

$
\Upsilon = (x_{11}, x_{21}, x_{12}, x_{22}, z_{11}, z_{12}, z_{21}, z_{22}, z_{31}, z_{32}, u_{11}, u_{12}, u_{21}, u_{22}, u_{31}, u_{32},
v_{111}, v_{121}, v_{112}, v_{122}, v_{211},$ $v_{221}, v_{212}, v_{222}, v_{311}, v_{321}, v_{312}, v_{322}, \zeta_{111},  \zeta_{121},  \zeta_{112},  \zeta_{122},  \zeta_{211}, \zeta_{221},  \zeta_{212},  \zeta_{222},  \zeta_{311},  \zeta_{321},  \zeta_{312},  \zeta_{322},  w_{11}, w_{12},$  $w_{13}, w_{21}, w_{22}, w_{23}, w_{31}, w_{32}, w_{33}, t_1, t_2, t_3, \theta_1, \theta_2, \theta_3)$

We apply the symmetry results to each of the subsets of variables where the facility appears, i.e.:

\begin{align}
\{x_{1k}, x_{2k}\}\; k=1,2, \tag{{\rm SUBSET $x_{\cdot k}$}}
\end{align}
\begin{align}
\{z_{i1}, z_{i2}\}\;i=1,2,3, \tag{{\rm SUBSET $z_{i\cdot}$}}
\end{align}
\begin{align}
\{u_{i1}, u_{i2}\}\; i=1,2,3, \tag{{\rm SUBSET $u_{i\cdot}$}}
\end{align}
\begin{align}
\{v_{i1k}, v_{i2k}\}\; i=1,2,3, k=1,2, \tag{{\rm SUBSET $v_{i\cdot k}$}}
\end{align}
\begin{align}
\{\zeta_{i1k}, \zeta_{i2k}\} ; i=1,2,3, k=1,2, \tag{{\rm SUBSET $\zeta_{i\cdot k}$}}
\end{align}

Let us construct the symmetry adapted basis for any of the above  subsets of variables, $\{Y_1, Y_2\}$. This construction will be applied to all the above $20$ sets of variables.

First, the components in the symmetry-adapted basis are indexed by the partitions $\lambda \vdash (2)$, thus $\lambda \in \{(2), (1,1)\}$. The $\beta$ to be taken into account are: $\beta \in \{(0,0), (1,0), (2,0), (1,1)\}$ with shapes $\mu$ equal to $(2)$, $(1,1)$, $(1,1)$ and $(2)$, respectively. Thus, the semistandard generalized  Young tableaux for each of these shapes and contents are:
\begin{itemize}
\item $\mu=(2)$:$\begin{ytableau}
1 & 1
\end{ytableau}$ and $\begin{ytableau}
1 & 2
\end{ytableau}$
\item $\mu=(1,1)$: No semistandard generalized  Young tableaux exists in this case.
\end{itemize}
Hence, there is only one irreducible component in this case (for $\mu=(2)$), so the symmetry-adapted basis in $\R[Y_1, Y_2]$ is
$$
\{1, Y_1+Y_2, Y_1^2+Y_2^2, Y_1Y_2\}
$$
We observe that the standard monomial basis for this set of two variables has $6$ monomials while this basis has only four elements.

Now, applying this shape to the 13 blocks of variables {\rm( {\rm SUBSET $x_{\cdot k}$} )-({\rm SUBSET $\zeta_{i\cdot k}$})} and adding the standard monomial basis to the remainder variables ($w$ and $t$) we get that a $392$-elements symmetry-adapted basis with degree up to $2$ for our problem.

The moment matrix is then a $392\times 392$-size matrix whose elements are computed by applying the sym-linearization operator $L^{\rm sym}$ to the two-by-two product of the elements in the basis. Analogously to compute the localizing matrices. We consider the standard sym-linear operator that maps the variables that are in the same orbit of our action $\varphi_{i,j,\ell}$ to the same element. For instance, when multiplying $(x_{11}+ x_{21})$ by itself we get $x_{11}^2+ x_{21}^2 + 2x_{11}x_{21}$. Since $x_{11}$ and $x_{21}$ are in the same orbit (there is a permutation sigma that maps $x_{11}$ into $x_{21}$), we get that $L^{\rm sym}((x_{11}+ x_{21})\cdot (x_{11}+ x_{21})) = 2 y_{2,0,0, ...} + 2y_{1,1, 0,...}$.

Observe that the standard monomial basis for this problem has $861$ elements (note that the number of variables is $40$ and all the possible standard monomials with these variables with degree up to $2$ is $\binom{40+2}{2}$), reducing considerably the sizes of the matrices for the SDP problem. This reduction is even more drastical when the number of demands points increases. Note that the standard monomial bases for each of the subsets of variables where the symmetry is applied have $6$ elements and it is reduced to $4$ by symmetry. Since this reduction is applied to thirteen sets of variables and then those elements are multiplied by the standard monomial bases for $w$ and $t$ (that depends of $n$, which is in general  the larger parameter of the location problem), we get that such a reduction will allow to solve larger instances of the problem. To  illustrate the advantage of this symmetry reduction,  in this special case, just by making some basic computations to count elements, the number of elements in the symmetry-adapted basis for 2-median planar problems with $n$ demand points is:
$$
\dfrac{n^4 + 16 n^3 +71 n^2 + 43 n + 68}{2}
$$
while the number of elements in the standard basis is:
$$
\dfrac{n^4 + 28 n^3 + 207 n^2 + 154 n + 30}{2}
$$
being the difference of elements $6n^3 + 68 n^2 + 43 n +7$. For instance, for $1000$ demand points, the number of variables appearing in the SDP problem is reduced by $6,068,043,007$ elements.

\end{ex}

\section{Conclusions}

We propose a novel, effective mixed integer nonlinear programming formulation for the continuous multifacility ordered median location problem for any $\ell_\tau$-norm and in any dimension. This formulation provides a unified approach for dealing with a broad family of multifacility location problems which up to now were usually solved  only for some special cases and usually in low dimension. We also show that the problem can be solved by using tools from the Theory of Moments for polynomial optimization,
by approximating the solution up to any desired degree of accuracy. Furthermore, such an algebraic framework allows us to prove new dimensionality reduction for the problem based on the sparsity and the symmetry of the formulation.

\end{document}